\documentclass[2p]{article}
\usepackage{amssymb,amsmath,amsthm}
\usepackage{indentfirst}
\usepackage{exscale}
\usepackage{relsize}
\usepackage{color}
\usepackage{geometry}

\newtheorem{theorem}{Theorem}[section]

\theoremstyle{definition}
\newtheorem{definition}[theorem]{Definition}

\newtheorem{lemma}[theorem]{Lemma}

\UseRawInputEncoding
\theoremstyle{remark}
\newtheorem{remark}[theorem]{Remark}

\numberwithin{equation}{section}
\begin{document}

\title{Three solutions for a quasilinear inclusion systems driven by a nonstandard Laplacian operator in $\mathbb{R}^N$ }
\date{}
\author {\ Cuiling Liu$^{1,3}$,
\ Xingyong Zhang$^{1,2}$\footnote{Corresponding author, E-mail address: zhangxingyong1@163.com} \\
      {\footnotesize $^{1}$Faculty of Science, Kunming University of Science and Technology, Kunming, Yunnan, 650500, P.R. China.}\\
      {\footnotesize $^{1,2}$Research Center for Mathematics and Interdisciplinary Sciences, Kunming University of Science and Technology,}\\
 {\footnotesize Kunming, Yunnan, 650500, P.R. China.}\\
      {\footnotesize $^{3}$Key Laboratory of Interdisciplinary Applications in Mathematics and Intelligent Engineering,}\\
      {\footnotesize Yunnan Education Department, China, 650500.}\\
 }
 \date{}
 \maketitle

 \begin{center}
 \begin{minipage}{15cm}
 \small  {\bf Abstract:}
 This paper establishes the existence of three distinct weak solutions for a class of nonhomogeneous quasilinear inclusion systems, which are governed by locally Lipschitz functionals within Orlicz-Sobolev spaces over unbounded domains  $\mathbb{R}^{N}$.
 By employing a new three critical points theorem established by Wu-Zhou in \cite{WuXian2021}, we extend the nonsmooth critical point theory to handle scenarios with nonlinear growth conditions.
 Our analysis imposes a precise structural condition on the nonlinear term $H$, requiring it to exhibit superlinear growth while maintaining subcritical asymptotic behavior. This condition establishes a delicate balance between growth restrictions and functional analytic requirements.
The novelties of this work include the construction of suitable energy functionals that are consistent with both the structure of Orlicz-Sobolev spaces and the characteristics of nonlinear growth. Through rigorous variational analysis and careful estimation techniques, we establish the existence of three weak solutions under these generalized conditions, thereby significantly expanding the applicability of critical point methods in nonstandard function spaces.

\end{minipage}
 \end{center}
  \allowdisplaybreaks
 \vskip2mm

\section{Introduction}
Partial differential equations with discontinuous nonlinearities represent a class of non-smooth dynamical models with significant applications in mathematical physics, engineering optimization, and materials science. These equations are characterized by nonlinear terms exhibiting abrupt variations or piecewise-defined features, such as friction models in contact mechanics, free boundary conditions in phase transition problems, or switching-type constraints in control theory. Their rigorous mathematical treatment requires the development of advanced analytical tools that transcend the traditional smoothness theories governing classical PDE analysis.

Within this research framework, we present an existence theorem concerning three critical points for locally Lipschitz functionals within Orlicz-Sobolev spaces. This work particularly addresses  the multiplicity of solutions for a class of quasilinear elliptic systems governed by nonhomogeneous differential operators. The prototypical problem under investigation takes the following variational model:
\begin{equation}\label{eq5-0}
 \left\{
  \begin{array}{ll}
 -\mbox{div}(\phi_1(|\nabla{u}|)\nabla{u})+a_{1}(x)\phi_{1}(|u|)u
 \in f(x)-[h^{-}(x,u,v),h^{+}(x,u,v)], &x\in \mathbb{R}^N,\\
 -\mbox{div}(\phi_{2}(|\nabla{v}|)\nabla{v})+a_{2}(x)\phi_{2}(|v|)v
 \in f(x)-[k^{-}(x,u,v),k^{+}(x,u,v)], &x\in \mathbb{R}^N,\\
  u\in W^{1,\Phi_{1}}(\mathbb{R}^N),v\in W^{1,\Phi_{2}}(\mathbb{R}^N),
    \end{array}
 \right.
 \end{equation}
where $a_{1}(x)=V_{1}(x)-\lambda_{1}$ and $a_{2}(x)=V_{2}(x)-\lambda_{2}$ are continuous functions satisfying $a_{1}(x)>0$ and $a_{2}(x)>0$, $h,k:\mathbb{R}^N\times \mathbb{R} \times \mathbb{R} \rightarrow \mathbb{R}$
are locally bounded measurable functions not necessarily continuous and
\begin{align*}
&
   h^{-}(x,u_{1},v_{1})
=  \lim_{\delta\rightarrow 0}\text{ess}\inf\{h(x,u_{2},v_{2}):\;|u_{1}-u_{2}|<\delta,\;|v_{1}-v_{2}|<\delta\},
  \nonumber\\
&
   h^{+}(x,u_{1},v_{1})
=  \lim_{\delta\rightarrow 0}\text{ess}\sup\{h(x,u_{2},v_{2}):\;|u_{1}-u_{2}|<\delta,\;|v_{1}-v_{2}|<\delta\},
  \nonumber\\
&
   k^{-}(x,u_{1},v_{1})
=  \lim_{\delta\rightarrow 0}\text{ess}\inf\{k(x,u_{2},v_{2}):\;|u_{1}-u_{2}|<\delta,\;|v_{1}-v_{2}|<\delta\},
  \nonumber\\
&
   k^{+}(x,u_{1},v_{1})
=  \lim_{\delta\rightarrow 0}\text{ess}\sup\{k(x,u_{2},v_{2}):\;|u_{1}-u_{2}|<\delta,\;|v_{1}-v_{2}|<\delta\}.
\end{align*}
$h^{-},h^{+},k^{-},k^{+}$
are superposition measurable.
$\phi_i~(i=1, 2):(0,+\infty)\rightarrow(0,+\infty)$ are two functions
which satisfy the following conditions:
\begin{itemize}
\item[$(\phi_1)$] $\phi_i\in C^1(0,+\infty)$, $t\phi_i(t)\rightarrow0$ as
 $t\rightarrow0$, $t\phi_i(t)\rightarrow+\infty$ as
 $t\rightarrow+\infty$;
 \item[$(\phi_2)$] $t\rightarrow t\phi_i(t)$ are strictly increasing;
 \item[$(\phi_3)$] $1<l_i:=\inf_{t>0}\frac{t^2\phi_i(t)}{\Phi_i(t)}\leq \sup_{t>0}\frac{t^2\phi_i(t)}{\Phi_i(t)}=:m_i<\min\{N,l_{i}^{\ast}\}$, where $\Phi_i(t):=\int_{0}^{|t|}s\phi_i(s)ds, \ t\in \mathbb{R}$, $l_{i}^{\ast}=\frac{l_{i}N}{N-l_{i}}$;
 \item[$(\phi_4)$]
    there exist positive constants $C_{i,1}$ and $C_{i,2}$, $i=1,2$ such that
    $$
    C_{i,1}|t|^{l_i}\le \Phi_i(t)\le C_{i,2}|t|^{l_i},\ \ \forall |t|<1;
    $$
\end{itemize}
\par
The system \eqref{eq5-0} for the case where $h^{-}(x,u,v)=h^{+}(x,u,v):=H_u(x,u,v)$ and $k^{-}(x,u,v)=k^{+}(x,u,v):=H_v(x,u,v)$ are continuous functions becomes
\begin{equation}\label{eq5-1}
 \left\{
  \begin{array}{ll}
 -\mbox{div}(\phi_1(|\nabla{u}|)\nabla{u})+a_{1}(x)\phi_{1}(|u|)u
 = f(x)-H_u(x,u,v), &x\in \mathbb{R}^N,\\
 -\mbox{div}(\phi_{2}(|\nabla{v}|)\nabla{v})+a_{2}(x)\phi_{2}(|v|)v
 = f(x)-H_v(x,u,v), &x\in \mathbb{R}^N,\\
  u\in W^{1,\Phi_{1}}(\mathbb{R}^N),v\in W^{1,\Phi_{2}}(\mathbb{R}^N).
    \end{array}
 \right.
 \end{equation}

The study of such systems has been extensively pursued across various fields, including nonlinear elasticity, plasticity, and non-Newtonian fluids (for detailed discussions, see
\cite{Zhikov1987,Ruzicka2000,Fukagai1995,Fuchs2000}). In recent years, problems akin to system \eqref{eq5-1} have garnered significant attention from researchers. Notable contributions can be found in the works of \cite{wang2017,wang2017-2,wang2017-3,wang2017-4,Zhang2021,Zhang2022,Zhang2023,Liu-Shibo2019}, with further references therein.

Recently, there has been a surge of interest in establishing the existence and multiplicity of solutions for partial differential equations with discontinuous nonlinearities. These studies are crucial for analyzing mathematical models in physical and mechanical problems. Key applications include the obstacle problem, the Goldshtik problem for separated flows of an incompressible fluid, the phenomenon of superconductivity, and the Elenbaas problem of electric discharge origination. For detailed discussions, see \cite{Chang1980,Pavlenko2018,Potapov2011,Potapov2010,Sherman1960,VyKhoi2001,Yang1995,Santos2023} and the references therein.

In \cite{AlvesCO2015}, Alves et al investigated the existence of solutions for the following  multivalued elliptic equations:
\begin{equation}\label{eq5-2}
 -\mbox{div}(\phi_1(|\nabla{u}|)\nabla{u})
 \in \partial j(u)+\lambda h,\;\; x\in \Omega,
 \end{equation}
where $\Omega\subset \mathbb{R}^{N}$ is a bounded domain with smooth boundary $\partial \Omega$, $\lambda >0$ is a parameter, and $h$ is a measurable function. When the nonlinear term satisfies a critical growth condition, the authors obtained a solution for  equation \eqref{eq5-0} by using the Ekeland's variational method.
In \cite{WuXian2021}, Wu et al considered the existence of three solutions for a class of reaction-diffusion equations with discontinuous nonlinearities. The authors first established  a new existence theorem  for three critical points of  locally Lipschitz functional. Subsequently, by applying this new theorem, they demonstrated that the inclusion problem has at least three solutions under a superlinear growth condition.
In \cite{AlvesCO2022}, Alves et al considered the existence of multiple solutions for a class of noncooperative elliptic system. Under certain technical conditions, the authors showed that the problem admits infinitely many large-energy solutions using the nonsmooth Fountain theorem without a symmetry condition.

Motivated by the work in \cite{WuXian2021,AlvesCO2022}, this paper aims to extend the results of \cite{WuXian2021} to the quasilinear inclusion system \eqref{eq5-0}. We assume that $H$ exhibits sub-$p$ growth with respect to $(u,v)$ and that the perturbation term  $f$ is a continuous function. Our main result establishes the multiplicity of nontrivial solutions for system \eqref{eq5-0}. We will prove the following:
\par
\noindent
\begin{theorem}\label{t5-1}
  Assume that $(\phi_1)$--$(\phi_4)$   and the following conditions hold:
\begin{itemize}
\item[$(H_0)$]
$H(x,u,v)=\int_{0}^{u}h(x,t,v)dt+\int_{0}^{v}k(x,0,s)ds=\int_{0}^{u}h(x,t,0)dt+\int_{0}^{v}k(x,u,s)ds$, $\forall\;(x, u,v)\in  \mathbb{R}^{N}\times \mathbb{R}\times \mathbb{R}$,
$H(x,u,v)=0$ if and only if $u=v=0$;
\item[$(H_1)$] there exist two  constants $C_0>0$ such that
\begin{equation*}
 |h(x,u,v)|+|k(x,u,v)|\leq C_0\left(|u|^{p-1}+|v|^{p-1}\right)
 \end{equation*}
for all  $(x, u,v)\in  \mathbb{R}^{N}\times \mathbb{R}\times \mathbb{R}$, where
$p \in (\max\{l_{1},l_{2}\},\min\{l_{1}^{\ast},l_{2}^{\ast}\})$;
\item[$(H_2)$] there exists a  constant $\widetilde{C}_0>0$ such that, for all
$(u,v)\in  \mathbb{R}\times \mathbb{R}$ and $x\in  \mathbb{R}^{N}$,
\begin{equation*}
 |\varrho_{1}|+|\varrho_{2}|\leq \widetilde{C}_0\left(|u|^{p-1}+|v|^{p-1}\right)
 \end{equation*}
for all  $\varrho_{1} \in \left[h^{-}(x,u,v),h^{+}(x,u,v)\right]$ and
$\varrho_{2} \in \left[k^{-}(x,u,v),k^{+}(x,u,v)\right]$;
\item[$(H_3)$] for all
$(u,v)\in  \left(\mathbb{R}\times \mathbb{R}\right)\backslash \{(0,0)\}$ and $x\in  \mathbb{R}^{N}$,
\begin{equation*}
0< p H(x,u,v)\leq  u \varrho_{1} + v \varrho_{2}
 \end{equation*}
for all  $\varrho_{1} \in \left[h^{-}(x,u,v),h^{+}(x,u,v)\right]$ and
$\varrho_{2} \in \left[k^{-}(x,u,v),k^{+}(x,u,v)\right]$;
\item[$(H_4)$] there exists a  constant $C_1>0$ such that
\begin{equation*}
 H(x,u,v)\geq  C_1\left(|u|^{p}+|v|^{p}\right)
 \end{equation*}
for all  $|(u,v)|\leq \delta $;
\item[$(F_{0})$]
$f \in L^{q}(\mathbb{R}^{N})$, $\frac{1}{q}+\frac{1}{p}=1$;
\item[$(V_{0})$]
$V_{i} \in C(\mathbb{R}^{N},\mathbb{R})$,
$V_{i}(x)^{-1}\in L^{\frac{r}{1-r}}(\mathbb{R}^{N})$
\mbox{and} there exists a constant $C_{4}$ such that
$0<\frac{8}{C_{i,5}h_{i,3}^{l_{i}}}<\lambda_{i}<C_{4}<V_{i}(x)$,\; $i=1,2$, where $C_{i,5}=\min\left\{C_{i,1},\Phi_{i}(1)\right\}$, $C_{i,1}$ and $h_{i,3}$ are embedding constants.
\end{itemize}
 Then  system \eqref{eq5-0} possesses at least three distinct solutions.
\end{theorem}
\begin{remark}\label{t5-2}
There is the set of pair of functions $(h,k)$ that satisfy the conditions $(H_0)$-$(H_4)$. For instance,
choose $a>0$, $q_{1}>0$, $q_{2}>0$ such that $q_{1}+q_{2}=p$ and let
$$
h(x,u,v)=p\mathcal{H}_{e}(|u|-a)|u|^{p-2}u+q_{1}|u|^{q_{1}-2}u|v|^{q_{2}}+p|u|^{p-2}u,
$$
$$
k(x,u,v)=p|v|^{p-2}v+q_{2}|v|^{q_{2}-2}v|u|^{q_{1}},
$$
where $\mathcal{H}_{e}:\mathbb{R}\rightarrow \mathbb{R}$ denotes the Heaviside function, i.e.,
\begin{equation*}
 \mathcal{H}_{e}(t)
=
 \left\{
  \begin{array}{ll}
 0,\;\;\text{絞}\; t\leq 0,\\
1,\;\;\text{絞}\; t> 0.\\
    \end{array}
 \right.
 \end{equation*}
Then $h$ and $k$ satisfy $(H_0)$-$(H_4)$.
\end{remark}

\begin{remark}\label{t5-3}
By $(\phi_4)$ and Lemma \ref{lemma2.6}, we have
\begin{eqnarray}\label{3.1.14}
\Phi_{i}(t)\geq C_{i,5}|t|^{l_{i}}, \forall \; t\geq 0,
\end{eqnarray}
where $C_{i,5}=\min\left\{C_{i,1},\Phi_{i}(1)\right\}$.
\end{remark}

 \allowdisplaybreaks
\section{Preliminaries}\label{section 2}
First, we focus on a survey of concepts and results from Orlicz and Orlicz-Sobolev spaces that
will be used in the text.
\begin{definition}\cite{Adams2003}\label{definition2.1}
Let $b$ : $[0,+\infty)\rightarrow [0,+\infty)$ be a right continuous, monotone increasing function with
\begin{itemize}
\item[$\rm(1)$] $b(0)=0$;
\item[$\rm(2)$] $\lim_{t\rightarrow +\infty} b(t)=+\infty$;
\item[$\rm(3)$] $b(t)>0 $ whenever $t>0$.
\end{itemize}
Then the function defined on $\mathbb{R}$ by $B(t)=\int_{0}^{|t|}b(s)ds$ is called as an $N$-function.
\end{definition}
\par
By the definition of $N$-function $B$, it is obvious that $B(0)=0$ and $B$ is strictly convex. We recall that an $N$-function $B$ satisfies a $\Delta_2$-condition globally (or near infinity) if
$$
  \sup_{t>0}\frac{B(2t)}{B(t)}<+\infty \ \ \left(\mbox{ or } \limsup_{t\rightarrow \infty}\frac{B(2t)}{B(t)}<+\infty\right),
$$
which implies that there exists a constant $K>0$, such that $B(2t)\leq K B(t)$ for all $t\geq0$ (or $t\geq t_0>0$). We also state the equivalent form that $B$ satisfies a $\Delta_2$-condition globally (or near infinity) if and only if for any $c\geq 1$, there exists a constant $K_c>0$ such that $B(ct)\leq K_c B(t)$ for all $t\geq0$ (or $t\geq t_0>0$).
\par
\noindent
\begin{definition}\cite{Adams2003}\label{definition2.2}
For an $N$-function $B$, we define
$$ \widetilde{B}(t)=\int_{0}^{|t|}b^{-1}(s)ds, \quad t\in \mathbb{R},$$
 where $b^{-1}$ is the right inverse of the right derivative $b$ of $B$. Then $\widetilde{B}$ is an $N$-function
called as the complement of $B$.
\end{definition}
\par
It holds that Young's inequality (see \cite{Adams2003,M. M. Rao2002})
\begin{equation}\label{2.1.1}
st\leq B (s)+\widetilde{B}(t), \quad s, t\geq 0
\end{equation}
and the inequality (see \cite[Lemma A.2]{Fukagai2006})
 \begin{equation}\label{2.1.2}
 \widetilde{B}(b(t))\leq B(2t), \quad t\geq 0.
 \end{equation}
\par
Now, we recall the Orlicz space $L^{B}(\Omega)$ associated with $N$-function $B$.
The Orlicz space $L^{B}(\Omega)$ is the vectorial space of the measurable functions $u: \Omega\rightarrow \mathbb{R}$ satisfying
$$\int_{\Omega}B(|u|)dx<+\infty,$$
where $\Omega \subset \mathbb{R}^N$ is an open set. It was obvious that $L^{B}(\Omega)$ is a Banach space endowed with Luxemburg norm
$$
\|u\|_{B}:=\inf \left\{\lambda >0: \int_{\Omega}B \left(\frac{u}{\lambda}\right)dx\le 1\right\}.
$$
 The fact that $B$ satisfies $\Delta_2$-condition globally implies that
\begin{equation}\label{2.1.3+}
 u_n\rightarrow u \mbox{ in } L^{B}(\Omega) \Longleftrightarrow \int_{\Omega}B(u_n-u)dx\rightarrow 0.
 \end{equation}
 Moreover, a generalized type of H\"{o}lder's inequality (see \cite{Adams2003,M. M. Rao2002})
$$\left| \int_{\Omega}uvdx \right|\leq 2\|u\|_{\Phi}\|v\|_{\widetilde{\Phi}}, \quad \mbox{ for all } u \in L^{\Phi}(\Omega)\mbox{ and } v \in L^{\widetilde{\Phi}}(\Omega)$$
can be gained by applying Young's inequality \eqref{2.1.1}.
\par
The corresponding Orlicz-Sobolev space (see \cite{Adams2003,M. M. Rao2002}) is defined by
$$
W^{1, B}(\Omega):=\left\{u \in L^{B}(\Omega): \frac{\partial u}{\partial x_i} \in L^{B}(\Omega), i=1,\cdots, N\right\}
$$
with the norm
$$\|u\|:=\|u\|_{B}+\|\nabla u \|_{B}.$$
\par
When $\Omega=\mathbb{R}^N$, the Orlicz-Sobolev space $W^{1, B}(\mathbb{R}^N)$ is the completion of $C_{0}^{\infty}(\mathbb{R}^N)$ under the norm
$$\|u\|:=\|u\|_{B}+\|\nabla u \|_{B}.$$
\par
In what follows,  we recall some inequalities and theorems which we will use. For more details, we refer the reader to the references \cite{Adams2003,Fukagai2006,Liu-Shibo2019}.
\par
\noindent
\begin{lemma}\cite{Adams2003,Fukagai2006}\label{lemma2.6}
If $B$ is an $N$-function, then the following conditions are equivalent:
\begin{itemize}
	\item[$\rm(1)$]
\begin{equation}\label{2.1.5}
1\leq l=\inf_{t>0}\frac{tb(t)}{B(t)}\leq\sup_{t>0}\frac{tb(t)}{B(t)}=m<+\infty;
\end{equation}
\item[$\rm(2)$] let $\zeta_0(t)=\min\{t^l, t^m\}$ and $\zeta_1(t)=\max\{t^l, t^m\}$, $t\geq0$. $B$ satisfies
$$\zeta_0(t)B(\rho)\leq B(\rho t)\leq \zeta_1(t)B(\rho), \quad \forall \rho, t\geq 0;$$
\item[$\rm(3)$] $B$ satisfies a $\Delta_2$-condition globally.
\end{itemize}
\end{lemma}
\noindent
\begin{lemma}\cite{Fukagai2006}\label{lemma2.7}
 If $B$ is an $N$-function and \eqref{2.1.5} holds, then $B$ satisfies
$$\zeta_0(\|u\|_{B})\leq\int_{\Omega}B(u)dx\leq\zeta_1(\|u\|_{B}), \quad \forall u \in L^{B}(\Omega).$$
\end{lemma}
\noindent
\begin{lemma}\cite{Fukagai2006}\label{lemma2.8}
 If $B$ is an $N$-function and \eqref{2.1.5} holds with $l>1$. Let $\widetilde{B}$ be the complement of $B$ and $\zeta_2(t)=\min\{t^{\widetilde{l}},t^{\widetilde{m}}\}$, $\zeta_3(t)=\max\{t^{\widetilde{l}},t^{\widetilde{m}}\}$ for $t\geq0$, where $\widetilde{l}:=\frac{l}{l-1}$ and $\widetilde{m}:=\frac{m}{m-1}$. Then $\widetilde{B}$ satisfies
\begin{itemize}
	\item[$\rm(1)$]
$$\widetilde{m}=\inf_{t>0}\frac{t\widetilde{B}^{'}(t)}{\widetilde{B}(t)}\leq\sup_{t>0}\frac{t\widetilde{B}^{'}(t)}{\widetilde{B}(t)}=\widetilde{l};$$
\item[$\rm(2)$]
$$\zeta_2(t)\widetilde{B}(\rho)\leq \widetilde{B}(\rho t)\leq \zeta_3(t)\widetilde{B}(\rho), \quad \forall  \rho, t\geq 0;$$
\item[$\rm(3)$]
$$\zeta_2(\|u\|_{\widetilde{B}})\leq\int_{\Omega}\widetilde{B}(u)dx\leq\zeta_3(\|u\|_{\widetilde{B}}), \quad \forall u \in L^{\widetilde{B}}(\Omega).$$
\end{itemize}
\end{lemma}
\par
\noindent
\begin{remark}\label{remark 2.9}
By Lemma 2.3 and Lemma 2.5, assumptions $(\phi_1)$--$(\phi_3)$ show that $\Phi_i ~(i=1, 2)$ and $\widetilde{\Phi}_i ~(i=1, 2)$ are $N$-functions satisfying $\Delta_2$-condition globally. Thus $L^{\Phi_i}(\mathbb{R}^N) (i=1, 2)$ and $W^{1, \Phi_i}(\mathbb{R}^N) (i=1, 2)$ are separable and reflexive Banach spaces (see \cite{Adams2003,M. M. Rao2002}).
\end{remark}
Let $\Psi$ be an $N$-function verifying $\Delta_2$-condition. If
 \begin{equation}\label{2.1.6}
\overline{\lim_{t\rightarrow 0}}\frac{\Psi(t)}{B(t)}< +\infty\;\;\;\mbox{and}\;\;\;
\mathop{\overline{\lim}}_{|t|\rightarrow +\infty}\frac{\Psi(t)}{B_{\ast}(t)}< +\infty,
\end{equation}
then we have a continuous embedding $W^{1,B}(\mathbb{R}^{N})\hookrightarrow L^{\Psi}(\mathbb{R}^{N})$.
Moreover, if
 \begin{equation}\label{2.1.7}
\lim_{|t|\rightarrow 0}\frac{\Psi(t)}{B(t)}< +\infty\;\;\;\mbox{and}\;\;\;
\lim_{|t|\rightarrow +\infty}\frac{\Psi(t)}{B_{\ast}(t)}=0,
 \end{equation}
then the embedding $W^{1,B}(\mathbb{R}^{N})\hookrightarrow L_{loc}^{\Psi}(\mathbb{R}^{N})$ is compact and we call that such $\Psi$ satisfies the subcritical condition.
\par
If we assume the following hypotheses for $V(x)$.
\begin{itemize}
\item[$(V)$]
$V \in C(\mathbb{R}^{N},\mathbb{R})$ satisfies
$V_{0}:=\inf_{x\in \mathbb{R}^{N}}V(x)>0$.
\end{itemize}
Then, we can consider the subspace $X$ of $W^{1,B}(\mathbb{R}^N)$ which defined by
\begin{eqnarray}\label{dddc1}
  X=\left\{u\in W^{1,B}(\mathbb{R}^N)\Big| \int_{\mathbb{R}^N}V(x)B(|u|)dx<\infty\right\}
\end{eqnarray}
  with the norm
\begin{eqnarray}\label{2.1.4}
  \|u\|_{1,B}=\|\nabla u\|_B+\|u\|_{B,V},
\end{eqnarray}
where
$$
  \|u\|_{B,V}=\inf\left\{\alpha>0\Big|\int_{\mathbb{R}^N} V(x)B\left(\frac{|u|}{\alpha}\right)dx\le 1\right\}.
$$
It is easy to see that $(X,\|\cdot\|)$ is a separable and reflexive Banach space (see \cite{Liu-Shibo2019}).
\noindent
\begin{lemma}\cite{Liu-Shibo2019}\label{lemma2.8.1}
 If $B$ is an $N$-function and \eqref{2.1.5} holds with $l>1$, $V(x)\in C(\mathbb{R}^{N})$,
 $V_{0}=\inf_{\mathbb{R}^{N}}V(x)>0$. Then for all $u\in X$ we have
$$\zeta_0(\|u\|_{B,V})\leq\int_{\mathbb{R}^{N}}V(x)B(u)dx\leq\zeta_1(\|u\|_{B,V}).$$
\end{lemma}
\vskip2mm
\par
We  point out certain useful properties regarding the norms on Orlicz-Sobolev space $X$ .
\noindent
\begin{lemma}\label{lemma2.8}
On $X$ the norms
$$
         \|u\|_{1,B}
  =    \|\nabla u\|_{B}+\|u\|_{B,V},
$$
$$
       \|u\|_{2,B}
  =    \max\left\{\|\nabla u\|_{B},\|u\|_{B,V}\right\}
$$
and
\begin{align*}
& \|u\|_{3,B}
  =    \inf\left\{\alpha_{1}>0\Big|
  \int_{\mathbb{R}^{N}} \left(B\left(\frac{|\nabla u|}{\alpha_{1}}\right)+V(x)B\left(\frac{|u|}{\alpha_{1}}\right)\right)dx\le 1\right\},
\end{align*}
are equivalent. More precisely, for every $u \in X$ we have
$$
  \|u\|_{3,B}\leq 2\|u\|_{2,B}\leq 2\|u\|_{1,B}\leq 4 \|u\|_{3,B}.
$$
\end{lemma}
\par
In the following we give some preliminaries  from nonsmooth critical point theory.
\par
Let $X$ be a real Banach space and $X^*$ the dual space of $X$.
We denote by $\langle\cdot,\cdot\rangle$ the duality pairing between $X^*$ and $X$.
A functional $f:X\rightarrow R$ is called locally Lipschitz if for each $u\in X$ there exist a neighborhood $U$ of $u$ and a constant $L>0$ such that
$$
|f(v)-f(w)|\leq L\|v-w\|,\;\forall\;v,w \in U.
$$
For any $u,v \in X$, we define the generalized directional derivative $f^{0}(u,v)$ of $f$ at point $u$ along the direction $v$ as
$$
f^{0}(u,v)=\limsup_{h\rightarrow 0,\lambda\downarrow 0}\frac{1}{\lambda}\left[f(u+h+\lambda v)-f(u+h)\right].
$$
The generalized gradient of the function $f$ at $u$, denoted by $\partial f(u)$, is the set
$$
\partial f(u)=\left\{w \in X^*:\; \langle w,v\rangle\leq f^{0}(u,v), \;\forall\;v \in X\right\}.
$$
Please see \cite{1981Chang} and \cite{Chang1986} for the properties of the generalized directional derivative and the generalized gradient.
\par
Set
$$
\lambda(u)=\min_{w\in \partial f(u)}\|w\|.
$$
If $\lambda(u)=0$, then $u$ is called a critical point of $f$. We shall use the following notations.
$$
K:=\{u\in X:\;\lambda(u)=0\},\;f_{c}:=\{u\in X:\;f(u)\leq c\},\;N_{\delta}(K):=\{u\in X:\;d(u,K)<\delta\}.
$$
\noindent
\begin{definition}\cite{WuXian2021}
Let $X$ be a reflexive Banach space and $f:X\rightarrow R$ a locally Lipschitz functional. We say that $f$ satisfies the condition (PS) if any sequence ${u_{n}}\subset X$ with $\lambda(u)\rightarrow 0$ and $\{f(u_{n})\}$ is bounded possesses a convergent subsequence.
\end{definition}
\par
\noindent
\begin{remark}
It is pointed in  \cite{Marano2002}  that $f$ satisfies the condition (PS) if and only if for any sequence ${u_{n}}\subset X$ if
$$
f^{0}(u_{n};v-u_{n})\geq -\varepsilon_{n}\|v-u_{n}\|,\;\forall\;v \in X,
$$
and ${f(u_{n})}$ is bounded, then ${u_{n}}$ possesses a convergent subsequence, where $\varepsilon_{n}\rightarrow 0^{+}$.
\end{remark}

\begin{lemma}\cite{WuXian2021}
Let $X$ be a reflexive Banach space with direct sum decomposition $X=X_{1}\bigoplus X_{2}$ and $0<dim X_{1}<+\infty$. Let $J:X\rightarrow R$ be a locally Lipschitz functional and bounded from below.
If there exists a $R>0$ and $r\in \mathbb{R}$ such that
$$
\alpha:=\max_{x\in S_{R}\cap X_{1}} J(x)<\beta:=\inf_{x\in X_{2}}J(x),\;\;\max_{x\in\overline{B}_{R}\cap X_{1}}J(x)<r
$$
and  $J$ satisfies the condition (PS) in $J^{-1}([c,r))$, then $J$ has at least three critical points, where $c:=\inf_{x\in X}J(x)$.
\end{lemma}
\par
\noindent

 \allowdisplaybreaks
\section{Main results and proofs}
Let $W_{1}$ denote the subspace of $W^{1,\Phi_1}(\mathbb{R}^{N})$ which is defined by
\begin{eqnarray*}
  W_{1}=\left\{u\in W^{1,\Phi_1}(\mathbb{R}^{N})\Big| \int_{\mathbb{R}^N}V_{1}(x)\Phi_{1}(|u|)dx<\infty\right\}
\end{eqnarray*}
endowed with the norm
\begin{eqnarray*}
  \|u\|_{1,\Phi_1}=\|\nabla u\|_{\Phi_1}+\|u\|_{\Phi_1,V_{1}},
\end{eqnarray*}
where
$$
  \|u\|_{\Phi_1,V_{1}}=\inf\left\{\alpha_{1}>0\Big|\int_{\mathbb{R}^N} V_{1}(x)\Phi_1\left(\frac{|u|}{\alpha_{1}}\right)dx\le 1\right\},
$$
and $W_{2}$ denote the subspace of $W^{1,\Phi_2}(\mathbb{R}^{N})$ which is defined by
\begin{eqnarray*}
  W_{2}=\left\{v\in W^{1,\Phi_2}(\mathbb{R}^{N})\Big| \int_{\mathbb{R}^N}V_{2}(x)\Phi_{2}(|v|)dx<\infty\right\}
\end{eqnarray*}
endowed with the norm
\begin{eqnarray*}
  \|v\|_{1,\Phi_2}=\|\nabla v\|_{\Phi_2}+\|v\|_{\Phi_2,V_{2}},
\end{eqnarray*}
where
$$
  \|v\|_{\Phi_2,V_{2}}
=
  \inf\left\{\alpha_{2}>0\Big|\int_{\mathbb{R}^N} V_{2}(x)\Phi_2\left(\frac{|v|}{\alpha_{2}}\right)dx\le 1\right\}.
$$
Then $W_{1}$ and $W_{2}$  are reflexive and separable Banach space.
In this paper, we consider the subspace $W=:W_{1}\times W_{2}$ of  Orlicz-Sobolev space $W^{1,\Phi_1}(\mathbb{R}^{N})\times W^{1,\Phi_2}(\mathbb{R}^{N})$ endowed with the norm
\begin{eqnarray}\label{d5-1}
        \|(u,v)\|
  =    \|u\|_{1,\Phi_1}+ \|v\|_{1,\Phi_2}
  =    \|\nabla u\|_{\Phi_1}+\|u\|_{\Phi_1,V_1}
     +\|\nabla v\|_{\Phi_2}+\|v\|_{\Phi_2,V_2}.
\end{eqnarray}
It is easy to see that  $(W,\|\cdot\|)$ is a separable and reflexive Banach space.
\par
As mentioned in Section 1, we will use the variational methods to find solutions to system \eqref{eq5-0}.
Therefore,
let the energy functional
 $I:W\rightarrow \mathbb{R}$ associated with system \eqref{eq5-0} is given by
 \begin{align}\label{d5-2}
&   I(u,v)
=   \int_{\mathbb{R}^{N}}\Phi_1(|\nabla u|)dx
  + \int_{\mathbb{R}^{N}}\Phi_2(|\nabla v|)dx
  + \int_{\mathbb{R}^{N}}(V_{1}(x)-\lambda_{1})\Phi_1(| u|)dx
  + \int_{\mathbb{R}^{N}}(V_{2}(x)-\lambda_{2})\Phi_2(| v|)dx
          \nonumber\\
&\qquad\qquad
  + \int_{\mathbb{R}^{N}}H(x,u,v)dx
  - \int_{\mathbb{R}^{N}}f(x)(u+v)dx,\quad (u,v)\in W.
\end{align}
For simplicity's sake, we use $J_{0}: W\rightarrow \mathbb{R}$ for
\begin{align}\label{d5-3}
&            J_{0}(u,v):
=           \int_{\mathbb{R}^{N}}\Phi_1(|\nabla u|)dx
          + \int_{\mathbb{R}^{N}}\Phi_2(|\nabla v|)dx
          + \int_{\mathbb{R}^{N}}(V_{1}(x)-\lambda_{1})\Phi_1(| u|)dx
          \nonumber\\
&\qquad\qquad\quad
          + \int_{\mathbb{R}^{N}}(V_{2}(x)-\lambda_{2})\Phi_2(| v|)dx
          - \int_{\mathbb{R}^{N}}f(x)(u+v)dx,
\end{align}
and $J_{1}: L^{p}(\mathbb{R}^{N})\times L^{p}(\mathbb{R}^{N})\rightarrow \mathbb{R}$ for
\begin{align}\label{d5-4}
J_{1}(u,v):=\int_{\mathbb{R}^{N}}H(x,u,v)dx,
\end{align}
which is well defined.
It is well known that $J_{0}\in C^{1}(W, \mathbb{R})$ with
\begin{align}\label{d5-5}
&            \langle J_{0}'(u,v), (\varphi,\psi)\rangle:
=           \int_{\mathbb{R}^{N}}\phi_1(|\nabla u|)\nabla u \nabla \varphi dx
          + \int_{\mathbb{R}^{N}}(V_{1}(x)-\lambda_{1})\phi_1(| u|)u \varphi dx
          + \int_{\mathbb{R}^{N}}\phi_2(|\nabla v|)\nabla v\nabla \psi dx
          \nonumber\\
&\qquad\qquad\qquad\qquad
          + \int_{\mathbb{R}^{N}}(V_{2}(x)-\lambda_{2})\phi_2(| v|)v \psi dx
          - \int_{\mathbb{R}^{N}}f(x)(\varphi+\psi)dx, \;(u,v),(\varphi,\psi)\in W.
\end{align}
Regarding the functional $J_{1}$, we can show using standard arguments that:

\begin{lemma}\label{}
Assume that $(H_0)$-$(H_2)$ hold. Considering $J_{1}$ on space $L^{p}(\mathbb{R}^{N})\times L^{p}(\mathbb{R}^{N})$, we have
\begin{itemize}
\item[(1)]
the functional $J_{1}: L^{p}(\mathbb{R}^{N})\times L^{p}(\mathbb{R}^{N})\rightarrow \mathbb{R}$ is locally Lipschitz in $L^{p}(\mathbb{R}^{N})\times L^{p}(\mathbb{R}^{N})$;
\item[(2)]
$\partial J_{1}(u,v) \subset [h^{-}(x,u,v),h^{+}(x,u,v)]\times [k^{-}(x,u,v),k^{+}(x,u,v)]$, almost everywhere $x\in \mathbb{R}^{N}$
   and $(u,v)\in L^{p}(\mathbb{R}^{N})\times L^{p}(\mathbb{R}^{N})$.
\end{itemize}
\end{lemma}
The above-mentioned result is similar to Theorem 4.1 in \cite{AlvesCO2014baohan} and Lemma 11 in \cite{Santos2023}, it establishes an important relation between the generalized gradients $\partial J_{1}(u,v)$ and $[h^{-}(x,u,v),h^{+}(x,u,v)]\times [k^{-}(x,u,v),k^{+}(x,u,v)]$,
and its proof follows the same steps, then we will omit its proof.

 Under the hypothesis $(H_1)$, there is $\widehat{C}_0>0$ such that
 \begin{align}\label{d5-6}
          H(x,u,v)
\leq      \widehat{C}_0 |u|^{p} + \widehat{C}_0|v|^{p}, \;\;\forall\;(u,v) \in \mathbb{R}\times \mathbb{R}\;\;\text{and}\;\;
                             x\in \mathbb{R}^{N}.
 \end{align}
 In addition, under the hypothesis $(H_3)$, for each $\varepsilon>0$, there exists  $c=c(\varepsilon)>0$ such that
  \begin{align}\label{d5-7}
          H(x,u,v)
\geq      \frac{c}{2} |u|^{p} + \frac{c}{2} |v|^{p}-\varepsilon(|u|^{2}+|v|^{2}), \;\;\text{a.e.\;in}\;(u,v) \in \mathbb{R}\times \mathbb{R} \;\;\text{and}\;\;
                             x\in \mathbb{R}^{N}.
 \end{align}
 From \eqref{d5-7} and $(H_4)$, there are $\widehat{C}_1>0$ such that
  \begin{align}\label{d5-8}
          H(x,u,v)
\geq      \widehat{C}_1  |u|^{p} + \widehat{C}_1|v|^{p}, \;\;\text{a.e.\;in}\;(u,v) \in \mathbb{R}\times \mathbb{R}\;\;\text{and}\;\;
                             x\in \mathbb{R}^{N}.
 \end{align}

We are going to justify that the inclusion below holds
\begin{align}\label{d5-9}
\partial J_{1}(u,v)\subset \left[h^{-}(x,u,v),h^{+}(x,u,v)\right]\times \left[k^{-}(x,u,v),k^{+}(x,u,v)\right]\;\;\text{a.e\;in}\;\mathbb{R}^{N},
\end{align}
where $(u,v)\in L^{p}(\mathbb{R}^{N})\times L^{p}(\mathbb{R}^{N})$.
This inclusion \eqref{d5-9} means that given
$(\rho_{1},\xi_{1})\in \partial J_{1}(u,v)\subset \left(L^{p}(\mathbb{R}^{N})\times L^{p}(\mathbb{R}^{N})\right)^{\ast}\approx L^{p'}(\mathbb{R}^{N})\times L^{p'}(\mathbb{R}^{N})$,
$\frac{1}{p}+\frac{1}{p'}=1$, there is
$(\widetilde{\rho}_{1},\widetilde{\xi}_{1})\in L^{p'}(\mathbb{R}^{N})\times L^{p'}(\mathbb{R}^{N})$ such that
\begin{align*}
\langle (\rho_{1},\xi_{1}),(u,v)\rangle
=
\int_{\mathbb{R}^{N}}\widetilde{\rho}_{1}udx+\int_{\mathbb{R}^{N}}\widetilde{\xi}_{1}vdx,\;\;\;
\forall\;(u,v)\in L^{p}(\mathbb{R}^{N})\times L^{p}(\mathbb{R}^{N}),
\end{align*}
where
\begin{align*}
(\widetilde{\rho}_{1},\widetilde{\xi}_{1})\in \left[h^{-}(x,u,v),h^{+}(x,u,v)\right]\times \left[k^{-}(x,u,v),k^{+}(x,u,v)\right]\;\;\text{a.e\;in}\;\mathbb{R}^{N}.
\end{align*}

In order to prove the inclusion mentioned above we will need some lemmas. First of all, by $(H_1)$, it is not difficult to verify that $H(x,\cdot,\cdot)\in \text{Lip}_{\text{loc}}(\mathbb{R}\times \mathbb{R},\mathbb{R})$ for all $x\in \mathbb{R}^{N}$. Moreover, note also that, for $(t,s)\in \mathbb{R}\times \mathbb{R}$, given
$(r_{1},h_{1}),\;(r_{2},h_{2})\in \mathbb{R}\times \mathbb{R}$ and $\gamma>0$
a direct computation gives
\begin{align*}
H^{\circ}(x,(t,s);(r_{1},h_{1}))
&\leq
\limsup\limits_{(r_{2},h_{2})\rightarrow 0,\;\gamma\downarrow 0}
\frac{1}{\gamma}\left(\int_{0}^{t+\gamma r_{1}+r_{2}}h(x,\tau,s+\gamma h_{1}+h_{2})d\tau
-\int_{0}^{t+r_{2}}h(x,\tau,s+h_{2})d\tau\right)
          \nonumber\\
&+
\limsup\limits_{(r_{2},h_{2})\rightarrow 0,\;\gamma\downarrow 0}
\frac{1}{\gamma}\left(\int_{0}^{s+\gamma h_{1}+h_{2}}k(x,0,\tau)d\tau
-\int_{0}^{s+h_{2}}k(x,0,\tau)d\tau\right)
\end{align*}
and
\begin{align*}
H^{\circ}(x,(t,s);(r_{1},h_{1}))
&\leq
\limsup\limits_{(r_{2},h_{2})\rightarrow 0,\;\gamma\downarrow 0}
\frac{1}{\gamma}\left(\int_{0}^{s+\gamma h_{1}+h_{2}}k(x,t+\gamma r_{1}+r_{2},\tau)d\tau
-\int_{0}^{s+h_{2}}k(x,t+r_{2},\tau)d\tau\right)
          \nonumber\\
&+
\limsup\limits_{(r_{2},h_{2})\rightarrow 0,\;\gamma\downarrow 0}
\frac{1}{\gamma}\left(\int_{0}^{t+\gamma r_{1}+r_{2}}h(x,\tau,0)d\tau
-\int_{0}^{t+r_{2}}h(x,\tau,0)d\tau\right).
\end{align*}
\par
\begin{lemma}\label{lemma3.1}
Assume that $(H_1)$ and $(H_2)$ hold. Then
\begin{align*}
\partial H(x,u,v)\subset \left[h^{-}(x,u,v),h^{+}(x,u,v)\right]\times \left[k^{-}(x,u,v),k^{+}(x,u,v)\right],
\;\;\;\forall\;u,v\in \mathbb{R},\;\;\forall\;x\in \mathbb{R}^{N}.
\end{align*}
\end{lemma}
\noindent
{\bf Proof.} For each $(u,v)\in \mathbb{R} \times \mathbb{R}$,
\begin{align*}
 \partial H(x,u,v)
=\left\{(u_{1},v_{1})\in \mathbb{R} \times \mathbb{R};\;
         \langle(u_{1},v_{1}),(r_{0},h_{0})\rangle \leq H^{\circ}(x,(u,v);(r_{0},h_{0})),\;\;
         \forall\;(r_{0},h_{0})\in \mathbb{R} \times \mathbb{R}
\right\}
\end{align*}
and so, given $(u_{1},v_{1})\in \partial H(x,u,v)$,  we get
\begin{align*}
u_{1}r_{1}+v_{1}h_{1}\leq H^{\circ}(x,(u,v);(r_{1},h_{1})),\;\;
         \forall\;(r_{1},h_{1})\in \mathbb{R} \times \mathbb{R}.
\end{align*}
Assume $r_{1}>0$ and $h_{1}=0$, then
\begin{align*}
u_{1}r_{1}
&\leq
H^{\circ}(x,(u,v);(r_{1},0))
\leq
\limsup\limits_{(r_{2},h_{2})\rightarrow 0,\;\gamma\downarrow 0}
\frac{1}{\gamma}\left(\int_{0}^{u+\gamma r_{1}+r_{2}}h(x,\tau,v+h_{2})d\tau
-\int_{0}^{u+r_{2}}h(x,\tau,v+h_{2})d\tau\right)
          \nonumber\\
&=
\limsup\limits_{(r_{2},h_{2})\rightarrow 0,\;\gamma\downarrow 0}
\frac{1}{\gamma}\int_{u+r_{2}}^{u+\gamma r_{1}+r_{2}}h(x,\tau,v+h_{2})d\tau
\leq
h^{+}(x,u,v)r_{1},
\end{align*}
and consequently,
\begin{align}\label{d5-10}
u_{1}\leq h^{+}(x,u,v),\;\;
         \forall\;(u,v)\in \mathbb{R} \times \mathbb{R}\;\text{and}\;x\in \mathbb{R}^{N}.
\end{align}
On the other hand,
\begin{align*}
u_{1}(-r_{1})
&\leq
H^{\circ}(x,(u,v);((-r_{1}),0))
\leq
\limsup\limits_{(r_{2},h_{2})\rightarrow 0,\;\gamma\downarrow 0}
\frac{1}{\gamma}\left(\int_{0}^{u+\gamma (-r_{1})+r_{2}}h(x,\tau,v+h_{2})d\tau
-\int_{0}^{u+r_{2}}h(x,\tau,v+h_{2})d\tau\right)
          \nonumber\\
&=
\limsup\limits_{(r_{2},h_{2})\rightarrow 0,\;\gamma\downarrow 0}
\frac{1}{\gamma}\left(-\left[\int_{0}^{u+r_{2}}h(x,\tau,v+h_{2})d\tau
-\int_{0}^{u+\gamma (-r_{1})+r_{2}}h(x,\tau,v+h_{2})d\tau\right]\right)
          \nonumber\\
&=
\limsup\limits_{(r_{2},h_{2})\rightarrow 0,\;\gamma\downarrow 0}
\frac{1}{\gamma}\left(\int_{u+\gamma (-r_{1})+r_{2}}^{u+r_{2}}-h(x,\tau,v+h_{2})d\tau\right)
\leq
h^{-}(x,u,v)(-r_{1}),
\end{align*}
that is,
\begin{align}\label{d5-11}
u_{1}\geq h^{-}(x,u,v),\;\;
         \forall\;(u,v)\in \mathbb{R} \times \mathbb{R}\;\text{and}\;x\in \mathbb{R}^{N}.
\end{align}
From \eqref{d5-10} and \eqref{d5-11},
\begin{align*}
h^{-}(x,u,v)\leq u_{1}\leq h^{+}(x,u,v),\;\;
         \forall\;(u,v)\in \mathbb{R} \times \mathbb{R}\;\text{and}\;x\in \mathbb{R}^{N}.
\end{align*}
On the other hand, if $r_{1}=0$ and $h_{1}>0$, we find
\begin{align*}
v_{1}h_{1}
&\leq
H^{\circ}(x,(u,v);(0,h_{1}))
\leq
\limsup\limits_{(r_{2},h_{2})\rightarrow 0,\;\gamma\downarrow 0}
\frac{1}{\gamma}\left(\int_{0}^{v+\gamma h_{1}+h_{2}}k(x,u+r_{2},\tau)d\tau
-\int_{0}^{v+h_{2}}k(x,u+r_{2},\tau)d\tau\right).
\end{align*}
In an entirely analogous way, it is shown that
\begin{align*}
k^{-}(x,u,v)\leq v_{1}\leq k^{+}(x,u,v),\;\;
         \forall\;(u,v)\in \mathbb{R} \times \mathbb{R}\;\text{and}\;x\in \mathbb{R}^{N}.
\end{align*}

\par
\begin{lemma}\label{lemma3.1}
Assume that $(H_1)$ and $(H_2)$ hold. Then
\begin{align*}
J_{1}^{\circ}((u,v);(u_{1},v_{1}))\leq  \int_{\mathbb{R}^{N}}H^{\circ}(x,(u,v);(u_{1},v_{1}))dx\;\;\; \forall\; (u,v),(u_{1},v_{1})\in L^{p}(\mathbb{R}^{N})\times L^{p}(\mathbb{R}^{N}).
\end{align*}
\end{lemma}
\noindent
{\bf Proof.}
By the definition of a generalized directional derivative,
\begin{align*}
    J_{1}^{\circ}((u,v);(u_{2},v_{2}))
=
\limsup\limits_{(r,h)\rightarrow 0,\;\gamma\downarrow 0}
\left[\int_{\mathbb{R}^{N}}
\frac{\left[H(x,(u,v)+\gamma(u_{2},v_{2})+(r_{0},h_{0}))-H(x,(u,v)+(r_{0},h_{0}))\right]dx}{\gamma}
\right],
\end{align*}
for $(u,v),(u_{2},v_{2})\in L^{p}(\mathbb{R}^{N})\times L^{p}(\mathbb{R}^{N})$.
Consider $(\gamma_{n})\subset \mathbb{R}_{+}$ and
$(r_{n},h_{n})\subset L^{p}(\mathbb{R}^{N})\times L^{p}(\mathbb{R}^{N})$ such that $(r_{n},h_{n})\rightarrow (0,0)$ in $L^{p}(\mathbb{R}^{N})\times L^{p}(\mathbb{R}^{N})$ and $\gamma_{n}\rightarrow 0^{+}$ in $\mathbb{R}$ with
\begin{align*}
\lim\limits_{n}
\left[
\int_{\mathbb{R}^{N}}
\frac{\left[H(x,(u,v)+\gamma_{n}(u_{2},v_{2})+(r_{n},h_{n}))-H(x,(u,v)+(r_{n},h_{n}))\right]dx}{\gamma_{n}}
\right]
=
J_{1}^{\circ}((u,v);(u_{2},v_{2})).
\end{align*}
Without loss of generality, we can assume $0<\gamma_{n}\leq 1$.
Setting
\begin{align*}
H_{n}(x)
=
\frac{\left[H(x,(u,v)+\gamma_{n}(u_{2},v_{2})+(r_{n},h_{n}))-H(x,(u,v)+(r_{n},h_{n}))\right]}{\gamma_{n}},
\end{align*}
there is a subsequence of $(F_{n})$, still denoted by itself, such that the Fatou's lemma combined with definition of $F^{\circ}(x,(u,v);(u_{2},v_{2}))$ gives
\begin{align*}
\limsup\limits_{n}\int_{\mathbb{R}^N}H_{n}(x)dx
&\leq
\int_{\mathbb{R}^N}\limsup\limits_{n}H_{n}(x)dx
          \nonumber\\
&=
\int_{\mathbb{R}^N}
\limsup\limits_{n}
\frac{\left[H(x,(u,v)+\gamma_{n}(u_{2},v_{2})+(r_{n},h_{n}))-H(x,(u,v)+(r_{n},h_{n}))\right]}{\gamma_{n}}
dx
          \nonumber\\
&\leq
\int_{\mathbb{R}^N}H^{\circ}(x,(u,v);(u_{2},v_{2}))dx.
\end{align*}
Once
\begin{align*}
\lim\limits_{n}\int_{\mathbb{R}^N}H_{n}(x)dx
=
J_{1}^{\circ}((u,v);(u_{2},v_{2}))dx,
\end{align*}
then
\begin{align*}
J_{1}^{\circ}((u,v);(u_{2},v_{2}))dx
\leq
\int_{\mathbb{R}^N}H^{\circ}(x,(u,v);(u_{2},v_{2}))dx,\;\;\;\forall\;(u,v),(u_{2},v_{2})\in L^{p}(\mathbb{R}^{N})\times L^{p}(\mathbb{R}^{N}).
\end{align*}

\par
\begin{lemma}\label{lemma3.1}
Assume that $(H_1)$ and $(H_2)$ hold. Then
\begin{align*}
\partial J_{1}(u,v)\subset \left[h^{-}(x,u,v),h^{+}(x,u,v)\right]\times \left[k^{-}(x,u,v),k^{+}(x,u,v)\right]
\;\;\; \text{a.e.}\; x\in \mathbb{R}^{N}
\end{align*}
with $(u,v)\in L^{p}(\mathbb{R}^{N})\times L^{p}(\mathbb{R}^{N})$.
\end{lemma}
{\bf Proof.}
Given $(u_{1},v_{1})\in L^{p}(\mathbb{R}^{N})\times L^{p}(\mathbb{R}^{N})$, if $(u_{2},v_{2})\in \partial J_{1}(u_{1},v_{1})$, then
\begin{align*}
(u_{2},v_{2})\in L^{p'}(\mathbb{R}^{N})\times L^{p'}(\mathbb{R}^{N})
\end{align*}
and
\begin{align*}
\langle(u_{2},v_{2}),(u_{3},v_{3})\rangle
\leq
J_{1}^{\circ}((u_{1},v_{1});(u_{3},v_{3}))
\;\;\; \forall\; (u_{3},v_{3})\in L^{p}(\mathbb{R}^{N})\times L^{p}(\mathbb{R}^{N}),
\end{align*}
that is,
\begin{align}\label{d5-12}
\int_{\mathbb{R}^{N}}u_{2}u_{3}dx
+\int_{\mathbb{R}^{N}}v_{2}v_{3}dx
\leq
J_{1}^{\circ}((u_{1},v_{1});(u_{3},v_{3}))
\;\;\; \forall\; (u_{3},v_{3})\in L^{p}(\mathbb{R}^{N})\times L^{p}(\mathbb{R}^{N}).
\end{align}
By Lemma \ref{}, we know that
\begin{align*}
J_{1}^{\circ}((u_{1},v_{1});(u_{3},v_{3}))
\leq
\int_{\mathbb{R}^{N}}H^{\circ}(x,(u_{1},v_{1});(u_{3},v_{3}))dx
\;\;\; \forall\; (u_{3},v_{3})\in L^{p}(\mathbb{R}^{N})\times L^{p}(\mathbb{R}^{N}).
\end{align*}
Thus, if $v_{3}=0$,
\begin{align*}
J_{1}^{\circ}((u_{1},v_{1});(u_{3},0))
\leq
\int_{\mathbb{R}^{N}}
\limsup\limits_{(r_{2},h_{2})\rightarrow 0,\;\gamma\downarrow 0}
\frac{1}{\gamma}\left(\int_{0}^{u_{1}+\gamma u_{3}+r_{2}}h(x,\tau,v_{1}+h_{2})d\tau
-\int_{0}^{u_{1}+r_{2}}h(x,\tau,v_{1}+h_{2})d\tau\right)dx,
\end{align*}
for all $u_{3}\in L^{p}(\mathbb{R}^{N})$.
Then, from \eqref{d5-12},
\begin{align*}
\int_{\mathbb{R}^{N}}u_{2}u_{3}dx
\leq
\int_{\mathbb{R}^{N}}
\limsup\limits_{(r_{2},h_{2})\rightarrow 0,\;\gamma\downarrow 0}
\frac{1}{\gamma}\left(\int_{0}^{u_{1}+\gamma u_{3}+r_{2}}h(x,\tau,v_{1}+h_{2})d\tau
-\int_{0}^{u_{1}+r_{2}}h(x,\tau,v_{1}+h_{2})d\tau\right)dx,
\end{align*}
for all $u_{3}\in L^{p}(\mathbb{R}^{N})$.
\par
Assuming $u_{3}(x)\geq 0$ a.e in $\mathbb{R}^{N}$, we see that
\begin{align*}
\int_{\mathbb{R}^{N}}u_{2}u_{3}dx
&\leq
\int_{\mathbb{R}^{N}}
\limsup\limits_{(r_{2},h_{2})\rightarrow 0,\;\gamma\downarrow 0}
\frac{1}{\gamma}\left(\int_{0}^{u_{1}+\gamma u_{3}+r_{2}}h(x,\tau,v_{1}+h_{2})d\tau
-\int_{0}^{u_{1}+r_{2}}h(x,\tau,v_{1}+h_{2})d\tau\right)dx
          \nonumber\\
&=
\int_{\mathbb{R}^{N}}
\left(\limsup\limits_{(r_{2},h_{2})\rightarrow 0,\;\gamma\downarrow 0}
\frac{1}{\gamma}\int_{u_{1}+r_{2}}^{u_{1}+\gamma u_{3}+r_{2}}h(x,\tau,v_{1}+h_{2})d\tau
\right)dx
          \nonumber\\
&\leq
\int_{\mathbb{R}^{N}}h^{+}(x,u_{1},v_{1})u_{3}dx,
\end{align*}
that is,
\begin{align*}
\int_{\mathbb{R}^{N}}u_{2}u_{3}dx
\leq
\int_{\mathbb{R}^{N}}h^{+}(x,u_{1},v_{1})u_{3}dx.
\end{align*}
On the other hand, if $u_{3}(x)\leq 0$ a.e in $\mathbb{R}^{N}$, we find
\begin{align*}
\int_{\mathbb{R}^{N}}u_{2}(-u_{3})dx
\leq
\int_{\mathbb{R}^{N}}h^{-}(x,u_{1},v_{1})(-u_{3})dx.
\end{align*}
Hence,
\begin{align}\label{d5-13}
\int_{\mathbb{R}^{N}}u_{2}u_{3}dx
\leq
\int_{\mathbb{R}^{N}}h^{+}(x,u_{1},v_{1})u_{3}dx,\;\;\;\text{if}\;\;u_{3}(x)\geq 0\;\;\text{a.e\;in}\;\mathbb{R}^{N}
\end{align}
and
\begin{align}\label{d5-14}
\int_{\mathbb{R}^{N}}u_{2}u_{3}dx
\geq
\int_{\mathbb{R}^{N}}h^{-}(x,u_{1},v_{1})u_{3}dx,\;\;\;\text{if}\;\;u_{3}(x)\leq 0\;\;\text{a.e\;in}\;\mathbb{R}^{N}.
\end{align}
In the following, we can see that
\begin{align*}
h^{-}(x,u_{1}(x),v_{1}(x))
\leq
u_{2}(x)
\leq
h^{+}(x,u_{1}(x),v_{1}(x))\;\;\;\text{a.e\;in}\;\;\mathbb{R}^{N}.
\end{align*}
\par
In fact, suppose, by contradiction, that there exists $\Omega\subset \mathbb{R}^{N}$ with $|\Omega|>0$ and
\begin{align}\label{d5-15}
u_{2}(x)
<
h^{-}(x,u_{1}(x),v_{1}(x)),\;\;\;\text{a.e\;in}\;\;\Omega.
\end{align}
Without loss of generality, we can assume that $|\Omega|<\infty$.
Thereby, $\overline{u}_{3}(x)=\chi_{\Omega}(x)\in L^{p}(\mathbb{R}^{N})$ and by \eqref{},
\begin{align*}
\int_{\Omega}u_{2}(x)dx
=
\int_{\mathbb{R}^{N}}u_{2}(x)\overline{u}_{3}(x)dx
\geq
\int_{\mathbb{R}^{N}}h^{-}(x,u_{1},v_{1})\overline{u}_{3}dx
=
\int_{\Omega}h^{-}(x,u_{1},v_{1})dx,
\end{align*}
which contradicts \eqref{d5-15}. Similarly, using \eqref{d5-13}, it turns out that
$u_{2}(x)\leq h^{+}(x,u_{1}(x),v_{1}(x))$ a.e in $\mathbb{R}^{N}$.
\par
In order to prove that
$k^{-}(x,u_{1}(x),v_{1}(x))\leq v_{2}(x) \leq k^{+}(x,u_{1}(x),v_{1}(x))$
a.e in $\mathbb{R}^{N}$, we can argue as above together with the condition  $(H_{0})$.
\par
Since the embedding $W \hookrightarrow L^{p}(\mathbb{R}^{N})\times L^{p}(\mathbb{R}^{N})$ is continuous and $W$ is dense in $L^{p}(\mathbb{R}^{N})\times L^{p}(\mathbb{R}^{N})$, it follows that
\begin{lemma}\label{}
Assume that $(H_0)$-$(H_2)$ hold. Considering $J_{1}$ on space $W$, we have
\begin{itemize}
\item[(1)]
$J_{1}$ is locally Lipschitz in $W$;
\item[(2)]
$\partial J_{1}(u,v) \subset [h^{-}(x,u,v),h^{+}(x,u,v)]\times [k^{-}(x,u,v),k^{+}(x,u,v)]$, almost everywhere $x\in \mathbb{R}^{N}$
   and $(u,v)\in W$.
\end{itemize}
\end{lemma}
\par
By basic results from nonsmooth analysis,
for each $(u,v)\in W$,
\begin{align}\label{d5-16}
     \partial I(u,v)
=    J_{0}'(u,v)
    + \partial J_{1}(u,v).
\end{align}
Namely, for each $(u,v)\in W$,
 and each $(\hat{u},\hat{v})\in \partial I$, there exists $(u_{1},v_{1})\in \partial J_{1}$ such that
\begin{align}\label{d5-17}
     \langle \partial I,(u_{0},v_{0})\rangle
=    \langle J_{0}'(u,v),(u_{0},v_{0})\rangle
    + \langle \partial J_{1}(u,v),(u_{0},v_{0})\rangle,\quad \forall\; (u_{0},v_{0})\in W,
\end{align}
i.e.,
\begin{align}\label{d5-18}
&  \langle(\hat{u},\hat{v}),(u_{0},v_{0})\rangle
=    \int_{\mathbb{R}^{N}}\phi_1(|\nabla u(x)|)\nabla u(x)\nabla u_{0}(x)dx
   + \int_{\mathbb{R}^{N}}(V_{1}(x)-\lambda_{1})\phi_1(|u(x)|)u(x)u_{0}(x)dx
          \nonumber\\
&\qquad\qquad\qquad\qquad
   + \int_{\mathbb{R}^{N}}\phi_2(|\nabla v(x)|)\nabla v(x)\nabla v_{0}(x)dx
   + \int_{\mathbb{R}^{N}}(V_{2}(x)-\lambda_{2})\phi_2(|v(x)|)v(x)v_{0}(x)dx
          \nonumber\\
&\qquad\qquad\qquad\qquad
    + \int_{\mathbb{R}^{N}}u_{1}(x)u_{0}(x)dx
    + \int_{\mathbb{R}^{N}}v_{1}(x)v_{0}(x)dx
    - \int_{\mathbb{R}^{N}}f(x)(u_{0}(x)+v_{0}(x))dx,\quad \forall\;(u_{0},v_{0})\in W.
\end{align}

\begin{lemma}\label{t5-10}
Assume that $(\phi_1)$-$(\phi_4)$, $(H_0)$-$(H_4)$, $(F_{0})$ and $(V_0)$ hold. Then $\lim\limits_{\|(u,v)\|\rightarrow \infty}I(u,v)=+\infty$ and $I$ is bounded from below.
\end{lemma}
\par
\noindent
{\bf Proof.}
For any given $(u,v)\in W$, by $(\phi_4)$, $(V)$  and  H\"{o}lder's inequality, we have
\begin{align}\label{d5-19}
&I(u,v)
=        \int_{\mathbb{R}^{N}}\Phi_1(|\nabla u|)dx
          + \int_{\mathbb{R}^{N}}\Phi_2(|\nabla v|)dx
          + \int_{\mathbb{R}^{N}}V_{1}(x)\Phi_1(| u|)dx
          + \int_{\mathbb{R}^{N}}V_{2}(x)\Phi_2(| v|)dx
          \nonumber\\
&~~~~
          - \lambda_{1}\int_{\mathbb{R}^{N}}\Phi_1(| u|)dx
          - \lambda_{2}\int_{\mathbb{R}^{N}}\Phi_2(| v|)dx
          + \int_{\mathbb{R}^{N}}H(x,u,v)dx
          - \int_{\mathbb{R}^{N}}f(x)(u+v)dx
          \nonumber\\
&\geq        \int_{\mathbb{R}^{N}}\Phi_1(|\nabla u|)dx
          + \int_{\mathbb{R}^{N}}\Phi_2(|\nabla v|)dx
          + \int_{\mathbb{R}^{N}}V_{1}(x)\Phi_1(| u|)dx
          + \int_{\mathbb{R}^{N}}V_{2}(x)\Phi_2(| v|)dx
          \nonumber\\
&~~~~
          - \frac{\lambda_{1}}{C_{4}}\int_{\mathbb{R}^{N}}V_{1}(x)\Phi_1(|u|)dx
          - \frac{\lambda_{2}}{C_{4}}\int_{\mathbb{R}^{N}}V_{2}(x)\Phi_2(|v|)dx
          + \widehat{C}_1\int_{\mathbb{R}^{N}}|u|^{p}dx
          + \widehat{C}_1\int_{\mathbb{R}^{N}}|v|^{p}dx
          \nonumber\\
&~~~~
          -\left(\int_{\mathbb{R}^{N}}|f(x)|^{q}dx\right)^{\frac{1}{q}}
           \left(\int_{\mathbb{R}^{N}}|u|^{p}dx\right)^{\frac{1}{p}}
          - \left(\int_{\mathbb{R}^{N}}|f(x)|^{q}dx\right)^{\frac{1}{q}}
           \left(\int_{\mathbb{R}^{N}}|v|^{p}dx\right)^{\frac{1}{p}}
          \nonumber\\
&\geq        \min\left\{1,1-\frac{\lambda_{1}}{C_{4}}\right\}
             \min\left\{\|u\|_{3,\Phi_1}^{l_{1}},\|u\|_{3,\Phi_1}^{m_{1}}\right\}
          - \|f\|_{q}\|u\|_{p}
         + \widehat{C}_1\|u\|_{p}^{p}
          \nonumber\\
&~~~~
          +\min\left\{1,1-\frac{\lambda_{2}}{C_{4}}\right\}
              \min\left\{\|v\|_{3,\Phi_2}^{l_{2}},\|v\|_{3,\Phi_2}^{m_{2}}\right\}
          - \|f\|_{q}\|v\|_{p}
         + \widehat{C}_1\|v\|_{p}^{p}
          \nonumber\\
&\geq
       \frac{\min\left\{1,1-\frac{\lambda_{1}}{C_{4}}\right\}}{2^{m_{1}}}
       \min\left\{\|u\|_{1,\Phi_1}^{l_{1}},\|u\|_{1,\Phi_1}^{m_{1}}\right\}
          - \|f\|_{q}C_{1,6}\|u\|_{1,\Phi_1}
         + \widehat{C}_1\|u\|_{p}^{p}
          \nonumber\\
&~~~~
+  \frac{\min\left\{1,1-\frac{\lambda_{2}}{C_{4}}\right\}}{2^{m_{2}}}
   \min\left\{\|v\|_{1,\Phi_2}^{l_{2}},\|v\|_{1,\Phi_2}^{m_{2}}\right\}
          - \|f\|_{q}C_{2,6}\|v\|_{1,\Phi_2}
         + \widehat{C}_1\|v\|_{p}^{p},
\end{align}
where $p\in \left(\max\left\{l_{1},l_{2}\right\},\min\left\{l_{1}^{\ast},l_{2}^{\ast}\right\}\right)$, $\frac{1}{p}+\frac{1}{q}=1$.
Consequently, as $\|(u,v)\|\rightarrow \infty$, distinguishing four possible cases ($\|u\|_{p}\rightarrow \infty$, $\|u\|_{p}$ is bounded, $\|v\|_{p}\rightarrow \infty$ or $\|v\|_{p}$ is bounded)
we get $I(u,v)\rightarrow +\infty$, and hence, together with \eqref{d5-19}, we know that $I$ is bounded from below.

\par
\noindent
\begin{lemma}\label{t5-11}
Assume that $(\phi_1)$-$(\phi_4)$, $(H_0)$-$(H_4)$, $(F_{0})$ and $(V_0)$ hold. Then $I$ satisfies the condition (PS).
\end{lemma}
\par
\noindent
{\bf Proof.}
Let $\{(u_{n},v_{n})\}\subset W$ be such that
\begin{align*}
|I(u_{n},v_{n})| \leq C, \;\;\forall\;n\in \mathbb{N}
\end{align*}
and
$\lim_{n\rightarrow \infty}\lambda(u_{n},v_{n})=0$.
By Lemma \ref{t5-10}, we know that $\{(u_{n},v_{n})\}$ is bounded.
Hence, up to a subsequence $\{(u_n,v_n)\}$,
there exists a point $(u,v)\in W$ such that\\
 $\star$\quad $u_n\rightharpoonup u$ in $W_{1}$,
 \quad $u_n\rightarrow u$ in $L^{p}(\mathbb{R}^{N})$,
 \quad $u_n(x)\rightarrow u(x)$ a.e. in $\mathbb{R}^{N}$;\\
 $\star$\quad $v_n\rightharpoonup v$ in $W_{2}$,
 \quad $v_n\rightarrow v$ in $L^{p}(\mathbb{R}^{N})$,
 \quad $v_n(x)\rightarrow v(x)$ a.e. in $\mathbb{R}^{N}$.
 \par
Since
$\lambda(u_{n},v_{n}):=\min_{(\hat{u},\hat{v})\in \partial I(u_{n},v_{n})}\|(\hat{u},\hat{v})\|$,
there exists $(\hat{u}_{n},\hat{v}_{n})\in \partial I(u_{n},v_{n})$ such that
$\lambda(u_{n},v_{n})=\|(\hat{u}_{n},\hat{v}_{n})\|$.
Thus, there exists $(\hat{u},\hat{v})\in \partial I(u,v)$,
$(\tilde{u}_{n},\tilde{v}_{n})\in [h^{-}(x,u_{n},v_{n}),h^{+}(x,u_{n},v_{n})]\times [k^{-}(x,u_{n},v_{n}),k^{+}(x,u_{n},v_{n})]$,
$(\hat{u},\hat{v})\in \partial I(u,v)$,
$(\tilde{u},\tilde{v})\in [h^{-}(x,u,v),h^{+}(x,u,v)]\times [k^{-}(x,u,v),k^{+}(x,u,v)]$
such that
\begin{align}\label{d5-20}
   o_{n}(1)
&=
   \langle (\hat{u}_{n}-\hat{u},\hat{v}_{n}-\hat{v}), (u_{n}-u,v_{n}-v)\rangle
          \nonumber\\
&=  \int_{\mathbb{R}^{N}}\left(\phi_1(|\nabla u_{n}|)\nabla u_{n}-\phi_1(|\nabla u|)\nabla u \right)
    (\nabla u_{n}-\nabla u)dx
 + \int_{\mathbb{R}^{N}}\left(\phi_2(|\nabla v_{n}|)\nabla v_{n}-\phi_2(|\nabla v|)\nabla v\right)
  (\nabla v_{n}-\nabla v)dx
          \nonumber\\
&\quad
 + \int_{\mathbb{R}^{N}}\left(V_{1}(x)-\lambda_{1}\right)
     \left(\phi_1(| u_{n}|)u_{n}-\phi_1(|u|)u\right)(u_{n}-u)dx
 + \int_{\mathbb{R}^{N}}\left(V_{2}(x)-\lambda_{2}\right)
     \left(\phi_2(| v_{n}|)v_{n}-\phi_2(| v|)v\right)(v_{n}-v)dx
          \nonumber\\
&\quad
 + \int_{\mathbb{R}^{N}}(\tilde{u}_{n}-\tilde{u})(u_{n}-u)dx
 + \int_{\mathbb{R}^{N}}(\tilde{v}_{n}-\tilde{v})(v_{n}-v)dx.
\end{align}
{\bf Conclusion.}
\begin{align}
\label{d5-21}&\int_{\mathbb{R}^{N}}(\tilde{u}_{n}-\tilde{u})(u_{n}-u)dx\rightarrow 0\;\text{as}\;n\rightarrow +\infty,\\
\label{d5-22}&\int_{\mathbb{R}^{N}}(\tilde{v}_{n}-\tilde{v})(v_{n}-v)dx\rightarrow 0\;\text{as}\;n\rightarrow +\infty.
\end{align}
In fact, due to H\"{o}lder's inequality
\begin{align*}
\left|\int_{\mathbb{R}^{N}}(\tilde{u}_{n}-\tilde{u})(u_{n}-u)dx\right|
&\leq
\int_{\mathbb{R}^{N}}\left|\tilde{u}_{n}-\tilde{u}\right|\left|u_{n}-u\right|dx
          \nonumber\\
&\leq
\left(\int_{\mathbb{R}^{N}}\left|\tilde{u}_{n}-\tilde{u}\right|^{\frac{p_{1}}{p_{1}-1}}dx\right)^{\frac{p_{1}-1}{p_{1}}}
\left(\int_{\mathbb{R}^{N}}\left|u_{n}-u\right|^{p_{1}}dx\right)^{\frac{1}{p_{1}}}
\end{align*}
On the other hand, from $(H_{2})$,
\begin{align*}
&     \left|\tilde{u}_{n}-\tilde{u}\right|^{\frac{p_{1}}{p_{1}-1}}
\leq \widetilde{C}_0^{\frac{p_{1}}{p_{1}-1}}
     \left(|u_{n}|^{p-1}+|u|^{p-1}+|v_{n}|^{p-1}+|v|^{p-1}\right)^{\frac{p_{1}}{p_{1}-1}}
          \nonumber\\
&\qquad\qquad\quad
\leq \left(2\widetilde{C}_0\right)^{\frac{p_{1}}{p_{1}-1}}
     \left(  |u_{n}|^{\frac{p_{1}(p-1)}{p_{1}-1}}+|u|^{\frac{p_{1}(p-1)}{p_{1}-1}}
            +|v_{n}|^{\frac{p_{1}(p-1)}{p_{1}-1}}+|v|^{\frac{p_{1}(p-1)}{p_{1}-1}}
     \right)
          \nonumber\\
&\qquad\qquad\quad
\leq \left(2\widetilde{C}_0\right)^{\frac{p_{1}}{p_{1}-1}}
     \left(  |u_{n}|^{\frac{l_{1}(p-1)}{p_{1}-1}}+|u_{n}|^{p_{1}}
           + |u|^{\frac{l_{1}(p-1)}{p_{1}-1}}+ |u|^{p_{1}}
           + |v_{n}|^{\frac{l_{2}(p-1)}{p_{1}-1}}+|v_{n}|^{p_{1}}
           + |v|^{\frac{l_{2}(p-1)}{p_{1}-1}}+|v|^{p_{1}}
     \right)
\end{align*}
where $p_{1} \in (\max\{m_{1},m_{2}\},\min\{l_{1}^{\ast},l_{2}^{\ast}\})$ and $p_{1} \geq p$,
and so,
\begin{align*}
&     \int_{\mathbb{R}^{N}}\left|\tilde{u}_{n}-\tilde{u}\right|^{\frac{p_{1}}{p_{1}-1}}dx
\leq \left(2\widetilde{C}_0\right)^{\frac{p_{1}}{p_{1}-1}}
     \int_{\mathbb{R}^{N}}\left(|u_{n}|^{p_{1}}+ |u|^{p_{1}}+|v_{n}|^{p_{1}}+|v|^{p_{1}}\right)dx
          \nonumber\\
&\qquad\qquad\quad\qquad\qquad\quad
   + \left(2\widetilde{C}_0\right)^{\frac{p_{1}}{p_{1}-1}}
     \int_{\mathbb{R}^{N}}
     \left(  |u_{n}|^{\frac{l_{1}(p-1)}{p_{1}-1}}
           + |u|^{\frac{l_{1}(p-1)}{p_{1}-1}}
           + |v_{n}|^{\frac{l_{2}(p-1)}{p_{1}-1}}
           + |v|^{\frac{l_{2}(p-1)}{p_{1}-1}}
     \right)
     dx
          \nonumber\\
&\qquad\qquad\quad\qquad\quad
=   \left(2\widetilde{C}_0\right)^{\frac{p_{1}}{p_{1}-1}}
     \int_{\mathbb{R}^{N}}\left(|u_{n}|^{p_{1}}+ |u|^{p_{1}}+|v_{n}|^{p_{1}}+|v|^{p_{1}}\right)dx
          \nonumber\\
&\qquad\qquad\quad\qquad\qquad\quad
   + \left(2\widetilde{C}_0\right)^{\frac{p_{1}}{p_{1}-1}}
     \int_{\mathbb{R}^{N}}V_{1}(x)^{-\frac{p-1}{p_{1}-1}}V_{1}(x)^{\frac{p-1}{p_{1}-1}}
         |u_{n}|^{\frac{l_{1}(p-1)}{p_{1}-1}}dx
          \nonumber\\
&\qquad\qquad\quad\qquad\qquad\quad
   + \left(2\widetilde{C}_0\right)^{\frac{p_{1}}{p_{1}-1}}
     \int_{\mathbb{R}^{N}}V_{1}(x)^{-\frac{p-1}{p_{1}-1}}V_{1}(x)^{\frac{p-1}{p_{1}-1}}
         |u|^{\frac{l_{1}(p-1)}{p_{1}-1}}dx
          \nonumber\\
&\qquad\qquad\quad\qquad\qquad\quad
   + \left(2\widetilde{C}_0\right)^{\frac{p_{1}}{p_{1}-1}}
     \int_{\mathbb{R}^{N}}V_{2}(x)^{-\frac{p-1}{p_{1}-1}}V_{2}(x)^{\frac{p-1}{p_{1}-1}}
         |v_{n}|^{\frac{l_{2}(p-1)}{p_{1}-1}}dx
          \nonumber\\
&\qquad\qquad\quad\qquad\qquad\quad
   + \left(2\widetilde{C}_0\right)^{\frac{p_{1}}{p_{1}-1}}
     \int_{\mathbb{R}^{N}}V_{2}(x)^{-\frac{p-1}{p_{1}-1}}V_{2}(x)^{\frac{p-1}{p_{1}-1}}
         |v|^{\frac{l_{2}(p-1)}{p_{1}-1}}dx
          \nonumber\\
&\qquad\qquad\quad\qquad\quad
\leq \left(2\widetilde{C}_0\right)^{\frac{p_{1}}{p_{1}-1}}
     \int_{\mathbb{R}^{N}}\left(|u_{n}|^{p_{1}}+ |u|^{p_{1}}+|v_{n}|^{p_{1}}+|v|^{p_{1}}\right)dx
          \nonumber\\
&\qquad\qquad\quad\qquad\qquad\quad
   + \left(2\widetilde{C}_0\right)^{\frac{p_{1}}{p_{1}-1}}
     \left(\int_{\mathbb{R}^{N}}V_{1}(x)^{-\frac{p-1}{p_{1}-p}}dx\right)^{\frac{p_{1}-p}{p_{1}-1}}
     \left(\int_{\mathbb{R}^{N}}V_{1}(x)|u_{n}|^{l_{1}}dx\right)^{\frac{p-1}{p_{1}-1}}
          \nonumber\\
&\qquad\qquad\quad\qquad\qquad\quad
   +   \left(2\widetilde{C}_0\right)^{\frac{p_{1}}{p_{1}-1}}
     \left(\int_{\mathbb{R}^{N}}V_{1}(x)^{-\frac{p-1}{p_{1}-p}}dx\right)^{\frac{p_{1}-p}{p_{1}-1}}
     \left(\int_{\mathbb{R}^{N}}V_{1}(x)|u|^{l_{1}}dx\right)^{\frac{p-1}{p_{1}-1}}
          \nonumber\\
&\qquad\qquad\quad\qquad\qquad\quad
   +   \left(2\widetilde{C}_0\right)^{\frac{p_{1}}{p_{1}-1}}
     \left(\int_{\mathbb{R}^{N}}V_{2}(x)^{-\frac{p-1}{p_{1}-p}}dx\right)^{\frac{p_{1}-p}{p_{1}-1}}
     \left(\int_{\mathbb{R}^{N}}V_{2}(x)|v_{n}|^{l_{2}}dx\right)^{\frac{p-1}{p_{1}-1}}
          \nonumber\\
&\qquad\qquad\quad\qquad\qquad\quad
   + \left(2\widetilde{C}_0\right)^{\frac{p_{1}}{p_{1}-1}}
     \left(\int_{\mathbb{R}^{N}}V_{2}(x)^{-\frac{p-1}{p_{1}-p}}dx\right)^{\frac{p_{1}-p}{p_{1}-1}}
     \left(\int_{\mathbb{R}^{N}}V_{2}(x)|v|^{l_{2}}dx\right)^{\frac{p-1}{p_{1}-1}}
          \nonumber\\
&\qquad\qquad\quad\qquad\quad
\leq \left(2\widetilde{C}_0\right)^{\frac{p_{1}}{p_{1}-1}}
     \int_{\mathbb{R}^{N}}\left(|u_{n}|^{p_{1}}+ |u|^{p_{1}}+|v_{n}|^{p_{1}}+|v|^{p_{1}}\right)dx
          \nonumber\\
&\qquad\qquad\quad\qquad\qquad\quad
   + \left(2\widetilde{C}_0\right)^{\frac{p_{1}}{p_{1}-1}}
     \|V_{1}^{-1}\|_{\frac{p-1}{p_{1}-p}}^{\frac{p-1}{p_{1}-1}}
     \left(\frac{1}{C_{1,5}}\int_{\mathbb{R}^{N}}V_{1}(x)\Phi_{1}\left(|u_{n}|\right)dx\right)^{\frac{p-1}{p_{1}-1}}
          \nonumber\\
&\qquad\qquad\quad\qquad\qquad\quad
   +  \left(2\widetilde{C}_0\right)^{\frac{p_{1}}{p_{1}-1}}
     \|V_{1}^{-1}\|_{\frac{p-1}{p_{1}-p}}^{\frac{p-1}{p_{1}-1}}
     \left(\frac{1}{C_{1,5}}\int_{\mathbb{R}^{N}}V_{1}(x)\Phi_{1}\left(|u|\right)dx\right)^{\frac{p-1}{p_{1}-1}}
          \nonumber\\
&\qquad\qquad\quad\qquad\qquad\quad
   +  \left(2\widetilde{C}_0\right)^{\frac{p_{1}}{p_{1}-1}}
     \|V_{2}^{-1}\|_{\frac{p-1}{p_{1}-p}}^{\frac{p-1}{p_{1}-1}}
     \left(\frac{1}{C_{2,5}}\int_{\mathbb{R}^{N}}V_{2}(x)\Phi_{2}\left(|v_{n}|\right)dx\right)^{\frac{p-1}{p_{1}-1}}
          \nonumber\\
&\qquad\qquad\quad\qquad\qquad\quad
   +  \left(2\widetilde{C}_0\right)^{\frac{p_{1}}{p_{1}-1}}
     \|V_{2}^{-1}\|_{\frac{p-1}{p_{1}-p}}^{\frac{p-1}{p_{1}-1}}
     \left(\frac{1}{C_{2,5}}\int_{\mathbb{R}^{N}}V_{2}(x)\Phi_{2}\left(|v|\right)dx\right)^{\frac{p-1}{p_{1}-1}}
          \nonumber\\
&\qquad\qquad\quad\qquad\quad
\leq \left(2\widetilde{C}_0\right)^{\frac{p_{1}}{p_{1}-1}}
     \left(\|u_{n}\|_{p_{1}}^{p_{1}}+ \|u\|_{p_{1}}^{p_{1}}+\|v_{n}\|_{p_{1}}^{p_{1}}+\|v\|_{p_{1}}^{p_{1}}\right)
          \nonumber\\
&\qquad\qquad\quad\qquad\qquad\quad
   + \left(2\widetilde{C}_0\right)^{\frac{p_{1}}{p_{1}-1}}
     \|V_{1}^{-1}\|_{\frac{p-1}{p_{1}-p}}^{\frac{p-1}{p_{1}-1}}
     \left(\frac{1}{C_{1,5}}\right)^{\frac{p-1}{p_{1}-1}}
     \max\left\{\|u_{n}\|_{\Phi_{1},V_{1}}^{\frac{l_{1}(p-1)}{p_{1}-1}},
                \|u_{n}\|_{\Phi_{1},V_{1}}^{\frac{m_{1}(p-1)}{p_{1}-1}}
         \right\}
          \nonumber\\
&\qquad\qquad\quad\qquad\qquad\quad
   +  \left(2\widetilde{C}_0\right)^{\frac{p_{1}}{p_{1}-1}}
     \|V_{1}^{-1}\|_{\frac{p-1}{p_{1}-p}}^{\frac{p-1}{p_{1}-1}}
     \left(\frac{1}{C_{1,5}}\right)^{\frac{p-1}{p_{1}-1}}
     \max\left\{\|u\|_{\Phi_{1},V_{1}}^{\frac{l_{1}(p-1)}{p_{1}-1}},
                \|u\|_{\Phi_{1},V_{1}}^{\frac{m_{1}(p-1)}{p_{1}-1}}
         \right\}
          \nonumber\\
&\qquad\qquad\quad\qquad\qquad\quad
   +  \left(2\widetilde{C}_0\right)^{\frac{p_{1}}{p_{1}-1}}
     \|V_{2}^{-1}\|_{\frac{p-1}{p_{1}-p}}^{\frac{p-1}{p_{1}-1}}
     \left(\frac{1}{C_{2,5}}\right)^{\frac{p-1}{p_{1}-1}}
     \max\left\{\|v_{n}\|_{\Phi_{2},V_{2}}^{\frac{l_{2}(p-1)}{p_{1}-1}},
                \|v_{n}\|_{\Phi_{2},V_{2}}^{\frac{m_{2}(p-1)}{p_{1}-1}}
         \right\}
          \nonumber\\
&\qquad\qquad\quad\qquad\qquad\quad
   +  \left(2\widetilde{C}_0\right)^{\frac{p_{1}}{p_{1}-1}}
     \|V_{2}^{-1}\|_{\frac{p-1}{p_{1}-p}}^{\frac{p-1}{p_{1}-1}}
     \left(\frac{1}{C_{2,5}}\right)^{\frac{p-1}{p_{1}-1}}
     \max\left\{\|v\|_{\Phi_{2},V_{2}}^{\frac{l_{2}(p-1)}{p_{1}-1}},
                \|v\|_{\Phi_{2},V_{2}}^{\frac{m_{2}(p-1)}{p_{1}-1}}
         \right\}
          \nonumber\\
&\qquad\qquad\quad\qquad\quad
:=\Pi
\end{align*}
Therefore,
\begin{align*}
\left|\int_{\mathbb{R}^{N}}(\tilde{u}_{n}-\tilde{u})(u_{n}-u)dx\right|
\leq
\left(\int_{\mathbb{R}^{N}}\left|\tilde{u}_{n}-\tilde{u}\right|^{\frac{p_{1}}{p_{1}-1}}dx\right)^{\frac{p_{1}-1}{p_{1}}}
\left(\int_{\mathbb{R}^{N}}\left|u_{n}-u\right|^{p_{1}}dx\right)^{\frac{1}{p_{1}}}
\leq
\Pi^{\frac{p_{1}-1}{p_{1}}}
\|u_{n}-u\|_{p_{1}}
\rightarrow 0.
\end{align*}
Now, we define the operators
 $\mathcal{F}: W_{1}\rightarrow(W_{1})^*$ by
 $$
 \langle\mathcal{F}(u),\tilde{u}\rangle
 := \int_{\mathbb{R}^N}\phi_1(|\nabla u|)\nabla u\nabla \breve{u}dx
  + \int_{\mathbb{R}^N}V_{1}(x)\phi_1(| u|) u \breve{u}dx, \;\;
 u, \breve{u} \in W_{1}
 $$
 and $\mathcal{G}: W_{2}\rightarrow(W_{2})^*$
 by
 $$
     \langle\mathcal{G}(v),\tilde{v}\rangle
 :=   \int_{\mathbb{R}^N}\phi_2(|\nabla v|)\nabla v\nabla \breve{v}dx
    + \int_{\mathbb{R}^N}V_{2}(x)\phi_2(| v|) v \breve{v}dx,
            \;\;v, \breve{v} \in W_{2}.
 $$
Then, $(V)$ and $(\phi_2)$ imply that operators $\mathcal{F}$ and $\mathcal{G}$ are strictly monotone.
So, it follows by  \eqref{d5-20},  \eqref{d5-21} and \eqref{d5-22} that
\begin{align*}
\int_{\mathbb{R}^{N}}\left(\phi_1(|\nabla u_{n}|)\nabla u_{n}-\phi_1(|\nabla u|)\nabla u \right)
    (\nabla u_{n}-\nabla u)dx
\rightarrow 0,
\end{align*}
\begin{align*}
\int_{\mathbb{R}^{N}}\left(\phi_2(|\nabla v_{n}|)\nabla v_{n}-\phi_2(|\nabla v|)\nabla v\right)
  (\nabla v_{n}-\nabla v)dx
\rightarrow 0,
\end{align*}
\begin{align*}
\int_{\mathbb{R}^{N}}\left(V_{1}(x)-\lambda_{1}\right)
     \left(\phi_1(| u_{n}|)u_{n}-\phi_1(|u|)u\right)(u_{n}-u)dx
\rightarrow 0,
\end{align*}
and
\begin{align*}
\int_{\mathbb{R}^{N}}\left(V_{2}(x)-\lambda_{2}\right)
     \left(\phi_2(| v_{n}|)v_{n}-\phi_2(| v|)v\right)(v_{n}-v)dx
\rightarrow 0,
\end{align*}
as $n\rightarrow \infty$.
The next discusses is similar to Lemma 3.9 in \cite{wang2017-4},
and its proof follows the same steps, then we will omit its proof.
Hence, $(u_{n},v_{n})\rightarrow (u,v)$ in $W$.
\par
\noindent

 Since $W$ is a reflexive and separable Banach spaces,  there exist two sequences  $\{e_{ij}: j\in \mathbb{N}\}\subset {W_{i}}~(i=1,2)$ and $\{e_{ij}^{*}: j\in {\mathbb{N}\}\subset {W_{i}}}^{*}~(i=1,2)$ such that
 \begin{equation*}\label{3.2.1}
 W_{i}=\overline{\mbox {span}\{e_{ij}: j=1,2,\cdots\}}, \quad {W_{i}}^{*}=\overline{\mbox{span}\{e_{ij}^{*}: j=1,2,\cdots\}},\quad i=1,2,
 \end{equation*}
 and
 \begin{equation*}\label{3.2.2}
 e^{*}_{in}(e_{im})=
 \begin{cases}
  \begin{array}{ll}
 1 &\mbox{if } n=m,\\
 0 &\mbox{if } n\neq m,\\
    \end{array}
 \end{cases}i=1,2.
 \end{equation*}
 Let $Y_{i(k)}$ and $Z_{i(k)}$ be the subsets of $W_{i}$ defined by
 \begin{equation*}\label{3.2.3}
 Y_{i(k)}:=\mbox{span}\{e_{ij}: j=1,\cdots,k\}, \quad Z_{i(k)}:=\overline{\mbox{span}\{e_{ij}: j=k+1,\cdots\}}, \quad i=1,2.
 \end{equation*}
  Then
 $$W_{i}=Y_{i(k)}\oplus Z_{i(k)},\quad i=1,2, \quad k\in\mathbb{N}.$$

In addition, for Banach space $W=W_{1}\times W_{2}$, there exists a sequence $\{\eta_{(j)}\}\subset W$ defined by
 \begin{equation*}\label{3.2.6}
 \eta_{(j)}=
 \begin{cases}
  \begin{array}{ll}
 (e_{1n},0) &\mbox{if }j=2n-1,\\
 (0,e_{2n}) &\mbox{if }j=2n,  \quad      \mbox{ for } n\in \mathbb{N},\\
    \end{array}
 \end{cases}
 \end{equation*}
 such that
 \begin{itemize}
 	\item[$\rm(1)$]
 $$W=\overline{\mbox{span}\{\eta_{(j)}: j=1,2,\cdots\}},$$
 \item[$\rm(2)$]
 $$W=Y_k\oplus Z_k,$$
 where
 $$Y_k:=\text{span}\{\eta_{(j)}: j=1,\cdots,k\} \quad \text{and} \quad Z_k:=\overline{\mbox{span}\{\eta_{(j)}: j=k+1,\cdots\}}.$$
 \end{itemize}
\par
\noindent
\begin{lemma}\label{lemma3.10}
Assume that $(\phi_1)$-$(\phi_4)$, $(H_0)$-$(H_4)$, $(F_{0})$ and $(V_0)$ hold. There exists $\varepsilon=\min\left\{\varepsilon_{1},\varepsilon_{2}\right\}>0$, such that $\|f\|_{q}<\varepsilon$. Then there exists $M:=(M_{1},M_{2})>0$,
such that $\alpha:=\max\limits_{x\in S_{M}\cap Y_{2k}}I(x)<\beta:=\inf\limits_{x\in Z_{2k}}I(x)$.
\end{lemma}
\par
\noindent
{\bf Proof.}
For any $(u,v)\in Z_{2k}$ with $\|u\|_{1,\Phi_1}=r_{1,k}$ and $\|v\|_{1,\Phi_2}=r_{2,k}$ we have
\begin{align}\label{3.1.18}
&I(u,v)
=        \int_{\mathbb{R}^{N}}\Phi_1(|\nabla u|)dx
          + \int_{\mathbb{R}^{N}}\Phi_2(|\nabla v|)dx
          + \int_{\mathbb{R}^{N}}V_{1}(x)\Phi_1(| u|)dx
          + \int_{\mathbb{R}^{N}}V_{2}(x)\Phi_2(| v|)dx
          \nonumber\\
&~~~~
          - \lambda_{1}\int_{\mathbb{R}^{N}}\Phi_1(| u|)dx
          - \lambda_{2}\int_{\mathbb{R}^{N}}\Phi_2(| v|)dx
          + \int_{\mathbb{R}^{N}}H(x,u,v)dx
          - \int_{\mathbb{R}^{N}}f(x)(u+v)dx
          \nonumber\\
&\geq        \int_{\mathbb{R}^{N}}\Phi_1(|\nabla u|)dx
          + \int_{\mathbb{R}^{N}}\Phi_2(|\nabla v|)dx
          + \int_{\mathbb{R}^{N}}V_{1}(x)\Phi_1(| u|)dx
          + \int_{\mathbb{R}^{N}}V_{2}(x)\Phi_2(| v|)dx
          \nonumber\\
&~~~~
          - \frac{\lambda_{1}}{C_{4}}\int_{\mathbb{R}^{N}}V_{1}(x)\Phi_1(|u|)dx
          - \frac{\lambda_{2}}{C_{4}}\int_{\mathbb{R}^{N}}V_{2}(x)\Phi_2(|v|)dx
          + \widehat{C}_1\int_{\mathbb{R}^{N}}|u|^{p}dx
          + \widehat{C}_1\int_{\mathbb{R}^{N}}|v|^{p}dx
          \nonumber\\
&~~~~
          -\left(\int_{\mathbb{R}^{N}}|f(x)|^{q}dx\right)^{\frac{1}{q}}
           \left(\int_{\mathbb{R}^{N}}|u|^{q}dx\right)^{\frac{1}{q}}
          - \left(\int_{\mathbb{R}^{N}}|f(x)|^{q}dx\right)^{\frac{1}{q}}
           \left(\int_{\mathbb{R}^{N}}|v|^{q}dx\right)^{\frac{1}{q}}
          \nonumber\\
&\geq        \min\left\{1,1-\frac{\lambda_{1}}{C_{4}}\right\}
             \min\left\{\|u\|_{3,\Phi_1}^{l_{1}},\|u\|_{3,\Phi_1}^{m_{1}}\right\}
          - \|f\|_{q}\|u\|_{q}
         + \widehat{C}_1\|u\|_{p}^{p}
          \nonumber\\
&~~~~
          +\min\left\{1,1-\frac{\lambda_{2}}{C_{4}}\right\}
              \min\left\{\|v\|_{3,\Phi_2}^{l_{2}},\|v\|_{3,\Phi_2}^{m_{2}}\right\}
          - \|f\|_{q}\|v\|_{q}
         + \widehat{C}_1\|v\|_{p}^{p}
          \nonumber\\
&\geq
       \frac{\min\left\{1,1-\frac{\lambda_{1}}{C_{4}}\right\}}{2^{m_{1}}}
       \min\left\{\|u\|_{1,\Phi_1}^{l_{1}},\|u\|_{1,\Phi_1}^{m_{1}}\right\}
          - \|f\|_{q}C_{1,6}\|u\|_{1,\Phi_1}
          \nonumber\\
&~~~~
+  \frac{\min\left\{1,1-\frac{\lambda_{2}}{C_{4}}\right\}}{2^{m_{2}}}
   \min\left\{\|v\|_{1,\Phi_2}^{l_{2}},\|v\|_{1,\Phi_2}^{m_{2}}\right\}
          - \|f\|_{q}C_{2,6}\|v\|_{1,\Phi_2}
          \nonumber\\
&\geq
       \min\left\{   \bar{d_{1}}\left(\frac{\|f\|_{q}C_{1,6}}{\bar{d_{1}}l_{1}}\right)^{\frac{l_{1}}{l_{1}-1}}
                   - \|f\|_{q}C_{1,6}\frac{\|f\|_{q}C_{1,6}}{\bar{d_{1}}l_{1}},
                     \bar{d_{1}}\left(\frac{\|f\|_{q}C_{1,6}}{\bar{d_{1}}m_{1}}\right)^{\frac{m_{1}}{m_{1}-1}}
                   - \|f\|_{q}C_{1,6}\frac{\|f\|_{q}C_{1,6}}{\bar{d_{1}}m_{1}}
           \right\}
          \nonumber\\
&~~~~
+       \min\left\{   \bar{d_{2}}\left(\frac{\|f\|_{q}C_{2,6}}{\bar{d_{2}}l_{2}}\right)^{\frac{l_{2}}{l_{2}-1}}
                   - \|f\|_{q}C_{2,6}\frac{\|f\|_{q}C_{2,6}}{\bar{d_{2}}l_{2}},
                     \bar{d_{2}}\left(\frac{\|f\|_{q}C_{2,6}}{\bar{d_{2}}m_{2}}\right)^{\frac{m_{2}}{m_{2}-1}}
                   - \|f\|_{q}C_{2,6}\frac{\|f\|_{q}C_{2,6}}{\bar{d_{2}}m_{2}}
           \right\},
\end{align}
In fact, let $\bar{d_{1}}:=\frac{\min\left\{1,1-\frac{\lambda_{1}}{C_{4}}\right\}}{2^{m_{1}}}$, we have
 $\bar{d_{2}}:=\frac{\min\left\{1,1-\frac{\lambda_{2}}{C_{4}}\right\}}{2^{m_{2}}}$
\begin{align}
      \bar{d_{1}}\|u\|_{1,\Phi_1}^{l_{1}}- \|f\|_{q}C_{1,6}\|u\|_{1,\Phi_1}
\geq
      \bar{d_{1}}\left(\frac{\|f\|_{q}C_{1,6}}{\bar{d_{1}}l_{1}}\right)^{\frac{l_{1}}{l_{1}-1}}
    - \|f\|_{q}C_{1,6}\frac{\|f\|_{q}C_{1,6}}{\bar{d_{1}}l_{1}},\\
      \bar{d_{1}}\|u\|_{1,\Phi_1}^{m_{1}}- \|f\|_{q}C_{1,6}\|u\|_{1,\Phi_1}
\geq
      \bar{d_{1}}\left(\frac{\|f\|_{q}C_{1,6}}{\bar{d_{1}}m_{1}}\right)^{\frac{m_{1}}{m_{1}-1}}
    - \|f\|_{q}C_{1,6}\frac{\|f\|_{q}C_{1,6}}{\bar{d_{1}}m_{1}},\\
      \bar{d_{2}}\|v\|_{1,\Phi_2}^{l_{2}}- \|f\|_{q}C_{2,6}\|v\|_{1,\Phi_2}
\geq
      \bar{d_{2}}\left(\frac{\|f\|_{q}C_{2,6}}{\bar{d_{2}}l_{2}}\right)^{\frac{l_{2}}{l_{2}-1}}
    - \|f\|_{q}C_{2,6}\frac{\|f\|_{q}C_{2,6}}{\bar{d_{2}}l_{2}},\\
      \bar{d_{2}}\|v\|_{1,\Phi_2}^{m_{2}}- \|f\|_{q}C_{2,6}\|v\|_{1,\Phi_2}
\geq
      \bar{d_{2}}\left(\frac{\|f\|_{q}C_{2,6}}{\bar{d_{2}}m_{2}}\right)^{\frac{m_{2}}{m_{2}-1}}
    - \|f\|_{q}C_{2,6}\frac{\|f\|_{q}C_{2,6}}{\bar{d_{2}}m_{2}},
\end{align}
On the other hand,
since any two norms in finite dimensional space are equivalent,  there exist positive constants $h_{1,1}$, $h_{1,2}$, $h_{2,1}$, $h_{2,2}$, $h_{1,3}$, $h_{1,4}$, $h_{2,3}$ and $h_{2,4}$  such that
 \begin{align}
 &h_{1,1}\|u\|_{1,\Phi_1}\leq\|u\|_{L^{p}}\leq h_{1,2}\|u\|_{1,\Phi_1},\quad \forall  u\in W_{1},\label{3.2.5}\\
 &h_{2,1}\|v\|_{1,\Phi_2}\leq\|v\|_{L^{p}}\leq h_{2,2}\|v\|_{1,\Phi_2},\quad \forall  v\in W_{2},\label{3.2.6}\\
  &h_{1,3}\|u\|_{1,\Phi_1}\leq\|u\|_{L^{l_{1}}}\leq h_{1,4}\|u\|_{1,\Phi_1},\quad \forall  u\in W_{1},\label{3.2.5-}\\
 &h_{2,3}\|v\|_{1,\Phi_2}\leq\|v\|_{L^{l_{2}}}\leq h_{2,4}\|v\|_{1,\Phi_2},\quad \forall  v\in W_{2},\label{3.2.6-}
 \end{align}
 where $h_{1,1}$, $h_{1,2}$, $h_{1,3}$, $h_{1,4}$ and $h_{2,1}$, $h_{2,2}$, $h_{2,3}$, $h_{2,4}$ depend on the spatial dimension of $W_{1}$ and $W_{2}$, respectively.
For any $(u,v)\in Y_{2k}$ with $\|u\|_{1,\Phi_1}=r_{1,k}$ and $\|v\|_{1,\Phi_2}=r_{2,k}$ we have
\begin{align}\label{3.1.21}
I(u,v)
&=        \int_{\mathbb{R}^{N}}\Phi_1(|\nabla u|)dx
          + \int_{\mathbb{R}^{N}}\Phi_2(|\nabla v|)dx
          + \int_{\mathbb{R}^{N}}V_{1}(x)\Phi_1(|u|)dx
          + \int_{\mathbb{R}^{N}}V_{2}(x)\Phi_2(|v|)dx
          \nonumber\\
&~~~~
          - \lambda_{1}\int_{\mathbb{R}^{N}}\Phi_1(|u|)dx
          - \lambda_{2}\int_{\mathbb{R}^{N}}\Phi_2(|v|)dx
          + \int_{\mathbb{R}^{N}}H(x,u,v)dx
          - \int_{\mathbb{R}^{N}}f(x)(u+v)dx
          \nonumber\\
&\leq        \max\left\{\|u\|_{3,\Phi_1}^{l_{1}},\|u\|_{3,\Phi_1}^{m_{1}}\right\}
         +  \widehat{C}_0\|u\|_{L^{p}}^{p}
         -  \lambda_{1} C_{1,5}\|u\|_{l_{1}}^{l_{1}}
         +\|f\|_{q}C_{1,6}\|u\|_{1,\Phi_1}
          \nonumber\\
&~~~~
         +  \max\left\{\|v\|_{3,\Phi_2}^{l_{2}},\|v\|_{3,\Phi_2}^{m_{2}}\right\}
         +  \widehat{C}_0\|v\|_{L^{p}}^{p}
         - \lambda_{2}C_{2,5}\|v\|_{l_{2}}^{l_{2}}
         +\|f\|_{q}C_{2,6}\|v\|_{1,\Phi_2}
          \nonumber\\
&\leq
           4\|u\|_{1,\Phi_1}^{l_{1}}
           +4\|u\|_{1,\Phi_1}^{m_{1}}
         +  \widehat{C}_0h_{1,2}^{p}\|u\|_{1,\Phi_1}^{p}
         -  \lambda_{1} C_{1,5}h_{1,3}^{l_{1}}\|u\|_{1,\Phi_1}^{l_{1}}
         +\|f\|_{q}C_{1,6}\|u\|_{1,\Phi_1}
          \nonumber\\
&~~~~
         +  4\|v\|_{1,\Phi_2}^{l_{2}}
         +4\|v\|_{1,\Phi_2}^{m_{2}}
         +  \widehat{C}_0h_{2,2}^{p}\|v\|_{1,\Phi_2}^{p}
         - \lambda_{2}C_{2,5}h_{2,3}^{l_{2}}\|v\|_{1,\Phi_2}^{l_{2}}
         +\|f\|_{q}C_{2,6}\|v\|_{1,\Phi_2}
          \nonumber\\
&\leq
\left(8-\lambda_{1} C_{1,5}h_{1,3}^{l_{1}}\right)
\|u\|_{1,\Phi_1}^{l_{1}}
+
\left(8-\lambda_{2}C_{2,5}h_{2,3}^{l_{2}}\right)
\|v\|_{1,\Phi_2}^{l_{2}}
          \nonumber\\
&~~~~
+ \left(4+ \widehat{C}_0h_{1,2}^{p}\right)\|u\|_{1,\Phi_1}^{p}
+\|f\|_{q}C_{1,6}\|u\|_{1,\Phi_1}
+ \left(4+ \widehat{C}_0h_{2,2}^{p}\right)\|v\|_{1,\Phi_2}^{p}
+\|f\|_{q}C_{2,6}\|v\|_{1,\Phi_2}
\end{align}
For $s>0$, set
\begin{align*}
     e_{11}(s)
:=
     \left(8-\lambda_{1} C_{1,5}h_{1,3}^{l_{1}}\right)s^{l_{1}}
   + \left(4+ \widehat{C}_0h_{1,2}^{p}\right)s^{p}
\end{align*}
and
\begin{align*}
     e_{12}(s)
:=
-\|f\|_{q}C_{1,6}s
+\min\left\{   \bar{d_{1}}\left(\frac{\|f\|_{q}C_{1,6}}{\bar{d_{1}}l_{1}}\right)^{\frac{l_{1}}{l_{1}-1}}
                   - \|f\|_{q}C_{1,6}\frac{\|f\|_{q}C_{1,6}}{\bar{d_{1}}l_{1}},
                     \bar{d_{1}}\left(\frac{\|f\|_{q}C_{1,6}}{\bar{d_{1}}m_{1}}\right)^{\frac{m_{1}}{m_{1}-1}}
                   - \|f\|_{q}C_{1,6}\frac{\|f\|_{q}C_{1,6}}{\bar{d_{1}}m_{1}}
           \right\}.
\end{align*}
Since $8-\lambda_{1} C_{1,5}h_{1,3}^{l_{1}}<0$ and $p>l_{1}$, $e_{11}(s)$ has a negative minimum at
$M_{1}:=\left(\frac{l_{1}(\lambda_{1} C_{1,5}h_{1,3}^{l_{1}}-8)}{p\left(4+ \widehat{C}_0h_{1,2}^{p}\right)}\right)^{\frac{1}{p-l_{1}}}$.
Moreover, notice that as $\|f\|_{q}\rightarrow 0$, $e_{12}(M_{1})\rightarrow 0$. Hence, there exists $\varepsilon_{1}>0$ such that as $\|f\|_{q}<\varepsilon_{1}$, one has $e_{11}(M_{1})< e_{12}(M_{1})$.
Similarly,  there exists $\varepsilon_{2}>0$ such that as $\|f\|_{q}<\varepsilon_{2}$, one has $e_{21}(M_{2})< e_{22}(M_{2})$.
Therefore, there exists $\varepsilon=\min\left\{\varepsilon_{1},\varepsilon_{2}\right\}>0$ such that as $\|f\|_{q}<\varepsilon$, one has $e_{11}(M_{1})+e_{21}(M_{2})<e_{12}(M_{1})+ e_{22}(M_{2})$.
The proof  is completed.

 \vskip2mm
\noindent{\bf Acknowledgements}
\par
This project is supported by Interdisciplinary Research Program of Kunming University of Science and Technology (No: KUST-xk202025006),  Natural Science Research Fund of Kunming University of Science and Technology (No: KKZ3202507048), and  Yunnan Fundamental Research Projects (grant No:202301AT070465).

 \vskip2mm
\noindent{\bf Conflict of Interest}
\par
The authors state no conflict of  interest.

\end{document}